\documentclass{amsart}
\usepackage{amssymb}
\usepackage{graphics}
\usepackage{graphicx}
\usepackage{latexsym}
\usepackage{amsmath}
\usepackage{amssymb}
\usepackage{amscd}
\usepackage{multicol}
\usepackage[cmtip,matrix,arrow]{xy}
 \usepackage{epsfig}

\newtheorem{thm}{Theorem}[section]
\newtheorem{prop}[thm]{Proposition}

\newtheorem{lem}[thm]{Lemma}

\theoremstyle{definition}
\newtheorem{dfn}[thm]{Definition}
\newtheorem{ex}[thm]{Example}

\theoremstyle{remark}
\newtheorem{rem}[thm]{Remark}

\newcommand{\RR}{\mathbb R}
\newcommand{\ZZ}{\mathbb Z}

\newcommand{\CC}{\mathbb C}
\newcommand{\HH}{\mathbb H}

\newcommand{\rk}{\mathrm{rank}\, }

\newcommand{\Ad}{\operatorname{Ad}}

\newcommand{\Sp}{{\mathrm{Sp}}}
\newcommand{\U}{{\mathrm{U}}}
\newcommand{\SU}{{\mathrm{SU}}}
\newcommand{\SO}{{\mathrm{SO}}}
\newcommand{\mfk}{\mathfrak{k}}

\newcommand{\mfg}{\mathfrak{g}}
\newcommand{\mfh}{\mathfrak{h}}
\newcommand{\mft}{\mathfrak{t}}

\newcommand{\mfs}{\mathfrak{s}}
\newcommand{\mfu}{\mathfrak{u}}

\newcommand{\Spin}{{\mathrm{Spin}}}

\newcommand{\depth}{\operatorname{depth}}

\newcommand{\Mmax}{M_{\max}}

\begin{document}
\address[O.~Goertsches]{Fachbereich Mathematik\\  Universit\"at Hamburg\\ Bundesstra\ss e 55 \\ 20146 Hamburg\\  Germany}
\email[]{oliver.goertsches@math.uni-hamburg.de}
\address[A.-L.~Mare]{Department of Mathematics and Statistics\\ University of Regina\\ Regina, Saskatchewan \\Canada S4S 0A2}
\email[]{mareal@math.uregina.ca}

\title{Non-abelian GKM Theory}

\author{Oliver Goertsches}
\author{Augustin-Liviu Mare}

\begin{abstract} We describe a generalization of GKM theory for actions of arbitrary compact connected Lie groups. To an action satisfying the non-abelian GKM conditions we attach a graph encoding the structure of the non-abelian 1-skeleton, i.e., the subspace of points with isotopy rank at most one less than the rank of the acting group. We show that the algebra structure of the equivariant cohomology can be read off from this graph. In comparison with ordinary abelian GKM theory, there are some special features due to the more complicated structure of the non-abelian 1-skeleton.

\vspace{0.5cm}

\end{abstract}
\maketitle

\section{Introduction}

Given an equivariantly formal action of a compact torus $T = (S^1)^r$ on a compact connected manifold $M$, the by now classical Chang-Skjelbred Lemma \cite[Lemma 2.3]{CS} states that the equivariant cohomology $H^*_T(M)$ of the action (with real coefficients) is, as an algebra over the ring $H^*(BT) = S(\mft^*)$, determined by its 1-skeleton, i.e., the subspace $M_{r-1,T}$ of points in $M$ whose isotropy group has codimension at most one in $T$. Important progress was  made by Goresky, Kottwitz, and MacPherson,
who, originally in the context of algebraic actions of tori on complex projective
varieties, found an explicit combinatorial description of $H^*_T(M)$ in terms of the $T$-action on $M_{r-1,T}$, see \cite[Theorem 7.2]{GKM}.
In fact, their result holds for any   $T$-action on $M$, under
the additional assumptions that the fixed point set of the $T$-action is finite and   $M_{r-1,T}$ is a union of $2$-spheres. (In the case of an orientable $M$, the last condition is equivalent to the pairwise independence of the weights of the isotropy representations at the fixed points). The relevant data can be encoded in a graph with vertices corresponding to the fixed points and edges corresponding to the $2$-spheres in $M_{r-1,T}$.

This approach to equivariant cohomology, nowadays referred to as GKM theory, has received much attention in the last years. For example, the interplay between the
combinatorics of the GKM graph and the geometry of $M$ has been investigated
in the case when $M$ is symplectic and $T$ acts in Hamiltonian fashion,
see, e.g., \cite{Gu-Za1}, \cite{Gu-Za2}, and also when
$M$ is a quotient $G/H$ of two compact connected Lie groups $G$ and $H$ of
equal rank and $T$  is a maximal torus in $H$, see \cite{GHZ}.     
 Various versions and generalizations of GKM theory were proposed: 
 Guillemin and Holm \cite{GH} allowed the group action to have  non-isolated 
fixed points (in the setting of Hamiltonian torus actions on symplectic manifolds). In \cite{GNT}, the first named author, together with Nozawa and T\"oben, considered GKM theory for Cohen-Macaulay torus actions, which are similar to equivariantly formal actions, but for which the fixed point set is replaced by the orbits of lowest dimension. 
Actions of arbitrary, possibly non-abelian, topological groups on certain stratified
spaces have been considered by Harada, Henriques, and Holm in 
\cite{Ha-He-Ho}. 

Our aim is to present a generalization of the classical GKM theory to the action of a non-abelian compact connected Lie group $G$ on a compact connected manifold $M$.  As already motivated by the Borel localization theorem, see, e.g., \cite[Theorem C.70]{GGK}, the natural replacement for the fixed point set in this context is the space $\Mmax$ of points whose isotropy group has rank equal to the rank of $G$. In Section \ref{section:one} we will show that the natural replacement of $M_{r-1,T}$ is the space $M_{r-1}$ of points whose isotropy rank is at most one smaller than the rank of $G$, in the sense that there is a non-abelian version of the Chang-Skjelbred Lemma involving this space, see Theorem \ref{thm:CS}. 

Having thus laid the theoretical foundation for our theory, we will, after reviewing (a slight generalization of) the classical GKM theory in Section \ref{abelian}, develop non-abelian GKM theory in Sections \ref{sec:non-abeliansk}-\ref{sec:non-abelianGKMgraphs}, under the so-called \emph{non-abelian GKM conditions}:
\begin{enumerate}
\item[(i)] The $G$-action is equivariantly formal.
\item[(ii)] The space $\Mmax$ consists of finitely many $G$-orbits.
\item[(iii)] For every $p\in \Mmax$, the weights of the isotropy representation of $G_p$ on $T_pM$ are pairwise linearly independent.
\end{enumerate}
We remark that these conditions are equivalent to the ordinary (abelian) GKM conditions for the induced action of a maximal torus of $G$.

Our main task will be to understand the structure of the space $M_{r-1}$. In the torus case, the GKM conditions imply that $M_{r-1,T}$ is always the union of two-dimensional manifolds which connect the fixed points, and on which the torus acts with one-dimensional maximal orbits (i.e., very simple cohomogeneity-one manifolds). Abstractly, these can be defined as the closures of the components of $M_{r-1,T}\setminus M^T$, where $M^T$ is the fixed point set. In the non-abelian setting, we will still consider the closures $N$ of the components of $M_{r-1}\setminus \Mmax$. Although these subspaces of $M$ are not necessarily smooth, see Remark \ref{rem:compnotsmooth}, the group $G$ acts on them with cohomogeneity-one in the sense that their $G$-orbit space is homeomorphic to either a closed interval or a circle. Any of the spaces $N$ above is of one of the following three types:
\begin{enumerate}
\item[(I)] The case which is most similar to that of a $2$-sphere in the 1-skeleton of a torus action is that of a space $N$ whose $G$-orbit space is a closed interval, with both endpoints of the interval corresponding to orbits in $\Mmax$.
\item[(II)] It can happen that the $G$-orbit space of $N$ is an interval, but only one of the endpoints corresponds to an orbit in $\Mmax$. In the torus case this phenomenon occurs only if one allows nonorientable manifolds; then, there can appear real projective two-planes in $M_{r-1,T}$, see Section \ref{abelian}. In the non-abelian setting this phenomenon is not connected to orientability, see Example \ref{u2}.
\item[(III)] The case when the $G$-orbit space of $N$ is a circle can only appear if $N$ is not smooth, see Example \ref{typec}. This case has no analogue for torus actions on manifolds.
\end{enumerate}
The GKM graph associated to a torus action contains a vertex for each fixed point and an edge for each two-sphere in the 1-skeleton connecting the two vertices which correspond to the two fixed points in the sphere. By analogy, the vertices of the nonabelian GKM graph to be constructed in Section \ref{sec:non-abelianGKMgraphs} should be related to the orbits of maximal rank, and the edges to the spaces $N$ above. However, this simple analogy is not sufficient: for example, we see from case (II) above that we need a second type of vertices representing those endpoints which correspond to orbits not contained in $\Mmax$. For the detailed construction of the non-abelian GKM graph, we refer to Section \ref{sec:non-abelianGKMgraphs}.

A final comment concerns the effectiveness of our method for
computing the $S(\mfg^*)^G$-algebra $H^*_G(M)$.  This can also be done by
 using the well-known formula
$$H^*_G(M) = H^*_T(M)^{W(G)},$$
where $T\subset G$ is a maximal torus and $W(G)$ the corresponding Weyl group.
As already pointed out, in our context the algebra $H^*_T(M)$ can be expressed by means of the classical GKM theory.  Since the action of $W(G)$ on this algebra can also be made  precise, the formula above leads to  concrete descriptions of $H^*_G(M)$.
Although such a description has in general a more complicated  appearance than the one which results from our approach,  it can, in principle,  be reduced to that one.
However, a considerable amount of work  is sometimes  needed to get this reduction. For a discussion of concrete examples we refer to Remark 
\ref{simpler}. Perhaps even more important than the simplicity of our presentations of the equivariant cohomology algebra is their ``naturality" in the sense that they reflect closely the geometry of the action. In other words, our main concern here is not
of combinatorial nature, but rather of geometric nature.

\section{A non-abelian Chang-Skjelbred Lemma}\label{section:one}
Let $G$ be a compact connected Lie group of rank $r$ acting on a compact connected manifold $M$. For each $i=0,\ldots, r$ we define $M_{i}:=\{p\in M\mid \rk G_p \geq i\}$. We also write $\Mmax:=M_{r}$.

Consider first the case when the Lie group $G$ is Abelian, i.e., a torus $T$ of rank $r$. Then, we have $\Mmax = M^T$, the fixed point set of the action. Assuming that the action is equivariantly formal, it is well-known that the restriction map $H^*_T(M)\to H^*_T(M^T)$ is injective. The Chang-Skjelbred Lemma \cite[Lemma 2.3]{CS} then states that its image is the same as the image of $H^*_T(M_{r-1})\to H^*_T(M^T)$. Put together, we have a short exact sequence
\begin{equation}\label{eq:changskjelbred}
0 \longrightarrow H^*_T(M) \longrightarrow H^*_T(M^T) \longrightarrow H^{*+1}_T(M_{r-1},M^T).
\end{equation}
In this section we derive a non-abelian version of this short exact sequence, see Theorem \ref{thm:CS} below. 

From now on, $G$ is again an arbitrary compact connected Lie group of rank $r$. 
We need some considerations concerning the (Krull) dimension of
certain ($G$-equivariant cohmology) $H^*(BG)$-modules, which  will be simply denoted by $\dim$. 
The following result is a slight generalization of \cite[Lemma 4.4]{FranzPuppe2003}.
It can be proven with exactly the same arguments as that result 
(the only new ingredient needed in the proof is the general result saying that
if $K$ is a compact, possibly non-connected, closed subgroup of $G$, then the
dimension of $H^*(BK)$ over $H^*(BG)$ is equal to the rank of $K$,
see, e.g.,  \cite[Proposition 2.6]{GM}).

\begin{lem}\label{lem:dimension} $\dim H^*_G(M,M_{i})\leq i-1$.
\end{lem}

We are now in a position to state and prove the main result of this section.

\begin{thm}\label{thm:CS} If an action of a compact connected Lie group $G$ on a compact connected manifold $M$ is equivariantly formal, then 
\[
0\longrightarrow H^*_G(M)\longrightarrow H^*_G(\Mmax)\longrightarrow H^{*+1}_G(M_{r-1},\Mmax)
\]
is exact.
\end{thm}
\begin{proof} The injectivity of $H^*_G(M)\to H^*_G(\Mmax)$ was proven in \cite[Proposition 4.8]{GR}. 
The exactness at $H^*_G(\Mmax)$ can be proven with the same method as
 the exactness of the Atiyah-Bredon sequence for torus actions \cite[Main Lemma]{Bredon}, see also \cite{FranzPuppe2003}. We decided to include the proof  here for the sake
 of completeness. Because 
\[
0\longrightarrow H^*_G(M)\longrightarrow H^*_G(\Mmax)\longrightarrow H^{*+1}_G(M,\Mmax)\longrightarrow 0
\]
is exact, we have $\depth H^*_G(M,\Mmax)\geq r-1$. On the other hand we have $\dim H^*_G(M,\Mmax)\leq r-1$ by Lemma \ref{lem:dimension}, so it follows that $H^*_G(M,\Mmax)$ has dimension $r-1$ (recall that, in general, depth cannot
exceed $\dim$). 

Since $\dim H^*_G(M,M_{r-1})\leq r-2$, the image of the natural map $H^*_G(M,M_{r-1})\to H^*_G(M,\Mmax)$ 
has dimension at most equal to  $r-2$, hence it is equal to 0 (since otherwise, by \cite[Lemma 4.3]{FranzPuppe2003}, it would have
dimension equal to $r-1$). It follows that 
\[
0\longrightarrow H^*_G(M,\Mmax)\longrightarrow H^*_G(M_{r-1},\Mmax)\longrightarrow H^{*+1}_G(M,M_{r-1})\longrightarrow 0
\]
is exact. Thus, an element in the kernel of $H^*_G(\Mmax)\to H^{*+1}_G(M_{r-1},\Mmax)$ must already lie in the kernel of $H^*_G(\Mmax)\to H^{*+1}_G(M,\Mmax)$ and hence is in the image of $H^*_G(M)\to H^*_G(\Mmax)$.
\end{proof}
\begin{rem} The theorem suggests that there might exist a non-abelian version of the Atiyah-Bredon sequence as in \cite{Bredon}. However this is unclear, as to copy the proof given there one needs to know that the modules $H^*_G(M_i,M_{i+1})$ are Cohen-Macaulay of dimension $i$. This is so far only proven for $i=r$: \cite[Corollary 4.5]{GR} shows that $H^*_G(M_r,M_{r+1})=H^*_G(\Mmax)$ is a free module over $H^*(BG)$.
\end{rem}

\section{Abelian GKM theory }\label{abelian}

In this section we review the theorem of Goresky, Kottwitz, and MacPherson \cite[Theorem 7.2]{GKM} (see also \cite[Section 11.8]{GuilleminSternberg}) concerning the description of the equivariant
cohomology of a torus action in terms of the 1-skeleton.  A compact, connected manifold $M$ with an action of a torus $T$ is a {\it GKM space} if:
\begin{enumerate}
\item[(i)] The $T$-action on $M$ is equivariantly formal.
\item[(ii)] The fixed point set $M^T$ is finite.
\item[(iii)] For every $p\in M^T$, the weights of the isotropy representation of $T$ on $T_pM$ are pairwise linearly
independent.
\end{enumerate}
The theory we present here is slightly more general than the classical one, in the sense that the manifold $M$ is not required to be orientable; hence the
connected components of the fixed point sets of (codimension-one) subtori of $T$ are not necessarily orientable.
Consequently, there might be two-dimensional submanifolds in the 1-skeleton that are not necessarily spheres which need to be taken into
account (see Lemma \ref{lett} and Example \ref{ex:gras} below): these new features are reflected in the resulting GKM graph by the presence of some stars, like for instance in 
Figure \ref{fig2}.

\subsection{The abelian 1-skeleton}\label{sabelian}
We set $r:=\rk T$ and consider the subspace 
$M_{r-1}=\{p\in M \mid \dim T_p \ge r-1\}$ of $M$, also
called the \emph{1-skeleton} of the action. It is the
union of the fixed point sets of finitely many codimension-one subtori of $T$. The following result, which is Theorem 11.6.1 in \cite{GuilleminSternberg}, states that any connected component of the fixed point set of a subtorus of $T$ contains a $T$-fixed point. We give another proof of this result here which does not use the notions of equivariant Thom and Euler class.
Note that the lemma is more general than needed here, i.e., it only requires the action to
be equivariantly formal. 

\begin{lem}\label{fixp} If the action of a torus $T$ on a compact connected manifold 
$M$ is
equivariantly formal and $H\subset T$ is a subtorus, then any connected component
of the fixed point set $M^H$ contains a $T$-fixed point.
\end{lem}

\begin{proof} Suppose that there exists  a connected component of $M^H$, call
it $N$, which does not contain any $T$-fixed point. We isolate $N$ from the
rest of $M^H$ by a $T$-equivariant open tubular neighborhood, call it $U$. 
We set $V:=M\setminus N$, which is also a $T$-invariant open subset of $M$.
By the argument in the proof of \cite[Theorem 11.4.1]{GuilleminSternberg}, the support of $H^*_T(U\cap V;\CC)$ 
is contained in the union of the complexified Lie algebras of all isotropy groups of points in $U\cap V$.
None of these groups contain $H$, hence (because $H$ is connected) the union above cannot contain $\mfh\otimes \CC$,
where $\mfh$ is the Lie algebra of $H$. In particular, $H^*_T(U\cap V;\CC)$ is a torsion
$H^*(BT; \CC)$-module. We now consider the Mayer-Vietoris sequence of the pair
$U, V$  in equivariant cohomology. The map $H^*_T(U\cap V;\CC) \to H^{*+1}_T(M;\CC)$
which  arises in this sequence is identically zero, since it is a homomorphism 
of $H^*(BT;\CC)$-modules whose domain is a torsion module and whose range is a free module. 
 Thus the Mayer-Vietoris sequence splits into short exact sequences:
 \begin{equation}\label{eq:gsproofnew}
 0\longrightarrow H^*_T(M;\CC)\longrightarrow H^*_T(U;\CC) \oplus H^*_T(V;\CC) \to H^*_T(U\cap V; \CC) \longrightarrow 0.
 \end{equation}
As all the fixed points of the action are contained in $V$, the natural restriction map $H^*_T(M;\CC)\to H^*_T(M^T;\CC)$ (which is injective by equivariant formality, see \eqref{eq:changskjelbred}) factors through $H^*_T(V;\CC)$. Thus, $H^*_T(M;\CC)\to H^*_T(V;\CC)$ is injective, which, via the short exact sequence \eqref{eq:gsproofnew}, implies that the restriction map $H^*_T(U;\CC) \to H^*_T(U\cap V; \CC)$ is injective. 
But then the support of $H^*_T(U\cap V; \CC)$ contains the 
support of $H^*_T(U;\CC)$, which in turn contains $\mfh \otimes \CC$
(because $H^*_T(U;\CC)\simeq H^*_T(N;\CC)\simeq H^*_{T/H}(N;\CC) \otimes S((\mfh \otimes \CC)^*))$). We obtained a contradiction, which finishes the proof. 
\end{proof}  

\begin{rem} The result stated in the lemma may fail to be true in the case when
the group $H$ is not connected. Consider for example the canonical $S^1$-action on
$\RR P^2$ and the subgroup $\ZZ_2=\{-1,1\}$ of $S^1$. One of the connected components 
of $(\RR P^2)^{\ZZ_2}$ is diffeomorphic to a circle and contains no $S^1$-fixed points.   
\end{rem}

From now on we will assume that the $T$-action on $M$
satisfies the GKM conditions above.
Using Lemma \ref{fixp}, one shows \cite[Theorem 11.8.1]{GuilleminSternberg} that any positive-dimensional component of a fixed point set of a codimension-one subtorus of $T$ is two-dimensional. We will refer to these components as the \emph{components of the 1-skeleton} $M_{r-1}$. The following lemma describes these submanifolds.

\begin{lem}\label{lett}  Let $T\times N \to N$ be a nontrivial action of a torus $T$ on a compact, connected,
two-dimensional manifold $N$ with nonempty fixed point set. Then:
\begin{enumerate}
\item[(a)] The kernel $K$ of the action is of codimension one in $T$.
\item[(b)] $N$ is diffeomorphic to $S^2$ or $\RR P^2$ in such a way that the action of the circle $T/K$ corresponds to the standard action of $S^1$ on $S^2$ respectively $\RR P^2$.
\end{enumerate}
In particular, the $T$-action has two or one fixed points, depending on whether $N$ is diffeomorphic to $S^2$ or $\RR P^2$.
\end{lem}

\begin{proof} Because the $T$-action has fixed points, it cannot be transitive. It is also nontrivial, which implies that it is of cohomogeneity-one. The result then follows readily from the classification of cohomogeneity-one actions of compact Lie groups on two dimensional compact
manifolds, see, e.g., \cite[Section 6]{Mo}. There are only three such actions that are effective, namely
the actions of $S^1$ on $S^2$, $\RR P^2$, respectively the Klein bottle $K^2$.
The first two actions are standard. To describe the third one, we recall that $K^2$
can be written as the orbit space $(S^1 \times S^1)/\ZZ_2$, where the $\ZZ_2$-action is given by
$(z_1, z_2)\mapsto (-z_1, z_2^{-1})$. The $S^1$-action on $K^2$ 
is given by multiplication on the first factor. This action has no fixed points, so we can rule it out.
\end{proof}

\subsection{The  GKM graph}
By the short exact sequence \eqref{eq:changskjelbred}, the map $H^*_T(M)\to H^*_T(M^T)$ is injective and
its image is the same as the image of the map $H^*_T(M_{r-1})\to H^*_T(M^T)$. 
To understand this image, we first need to consider the two special situations described in Lemma \ref{lett} and 
compute the corresponding equivariant cohomology algebras:
\begin{lem}\label{lem:cohompieces}
\begin{enumerate}
\item[(a)] Assume that a torus $T$ acts on the sphere $S^2$ with unique fixed points $p$ and $q$, such that a codimension-one subgroup $K\subset T$ acts trivially. 
Then the restriction map $H^*_T(S^2)\to H^*_T(\{p\})\oplus H^*_T(\{q\}) = S(\mft^*)\oplus S(\mft^*)$ defines an isomorphism between $H^*_T(M)$ and the space of all pairs $(f, g)\in S(\mft^*)\oplus S(\mft^*)$  such that $f|_{\mfk}=g|_{\mfk}$.
Here $\mfk$ is the Lie algebra of $K$.
\item[(b)] Assume that a torus $T$ acts on the projective space $\RR P^2$ with unique fixed point $p$. Then the restriction map $H^*_T(\RR P^2)\to H^*_T(\{p\})=S(\mft^*)$ is an isomorphism.
\end{enumerate}
\end{lem} 
 \begin{proof} Both assertions are consequences of the equivariant Mayer-Vietoris sequence; see also \cite[Theorem 11.7.2]{GuilleminSternberg} for case (a). In case (b), one considers tubular neighborhoods around $p$ and the unique exceptional orbit, and uses that the restriction map from the tubular neighborhood of the exceptional orbit to a regular orbit induces an isomorphism in cohomology. Note that both statements are special cases of \cite[Corollary 4.2]{GM}. \end{proof} 

The classical GKM graph encodes the structure of the 1-skeleton of an equivariantly formal torus action under the assumption that it consists of just a union of 2-spheres. 
In our more general situation, the graph has some new features, as follows.  There are two types of vertices, \emph{dots} and \emph{stars}. As in the orientable case, the dots correspond to the $T$-fixed points. Two dots $p$ and $q$ are joined by an edge if there exists a 2-sphere in the 1-skeleton whose $T$-fixed points are $p$ and $q$;  the edge is labeled with $K$,  the kernel of the $T$-action on the sphere. Each real projective two-plane in the 1-skeleton contains exactly one $T$-fixed point, say $p$, and one exceptional $T$-orbit. We draw a star representing the exceptional orbit, and an edge connecting this star with the dot which corresponds to $p$.

\begin{rem}If the dimension of $M$ is $2n$ (the codimension of the fixed point set is necessarily even, and, by assumption, the fixed point set is finite), then from each dot  emerge exactly $n$ edges, whereas from a star emerges always exactly one.
\end{rem}

The equivariant cohomology can be read off from the GKM graph as follows:
\begin{thm}\label{thm:gkmclassical}
 Via the natural restriction map 
 \[H^*_T(M)\to H^*_T(M^T) = \bigoplus_{p\in M^T} S(\mft^*),\] the equivariant cohomology $S(\mft^*)$-algebra $H^*_T(M)$ is isomorphic to the set of all tuples $(f_p)_{p\in M^T}$ with the property that if $p, q\in M^T$ are joined by an edge labeled
 with $K$, then $f_{p}|_{\mfk}=f_q|_{\mfk}.$ 
\end{thm}
\begin{proof} Using Lemma \ref{lem:cohompieces}, the same proof as in ordinary GKM theory applies, see \cite[Theorem 7.2]{GKM}. 
\end{proof}
Note that the components of the 1-skeleton which
are projective planes, that is, the edges connecting a dot with a star in the GKM graph, have no contribution to the equivariant cohomology. Nevertheless, we include them into the GKM graph for two reasons: firstly, we want the GKM graph to encode the entire structure of the 1-skeleton, and secondly, in the natural generalization of the GKM graph to actions of non-abelian groups to be defined below, the stars will in fact contribute nontrivially. Note that one could also label edges connecting dots with stars (and the stars themselves) with the respective isotropy groups, but this is not necessary because this data is of no relevance for the equivariant cohomology.

We end this section with an example of a GKM graph of a torus action on a non-orientable manifold.
 
 \begin{ex}\label{ex:gras} We consider the real Grassmannian $G_2(\RR^5)=G/K$, where $G:=\SO(5)$ and $K:={\rm S}({\rm O}(3)\times {\rm O}(2))$, with the action of 
$T:=\{1\} \times \SO(2)\times \SO(2)$, which is the standard maximal torus of $\SO(5)$. This Grassmannian
is not orientable: in general, $G_k(\RR^n)$ is orientable iff $n$ is even. We will show that our 
action satisfies the GKM assumptions (i)-(iii) above. Then we will construct the GKM graph and describe the equivariant
cohomology $S(\mft^*)$-algebra.

The group $K$ has two components,
  namely the identity component $K_0= \SO(3)\times \SO(2)$, along with $g_0K_0$, where $g_0$ is the diagonal matrix   ${\rm Diag}(1,1,-1,1, -1)$. 
  We consider the canonical projection map $\pi : G/K_0\to G/K$:
  it is the covering map that arises from the $\ZZ_2$-action on $G/K_0$ given by  multiplication
  by $g_0$   from the right.

We start with the GKM description for the $T$-action on $G/K_0$, by 
  following \cite{GHZ} (the theory derived there is applicable, as $G$ is a compact semisimple Lie group and $K_0$ a connected closed subgroup of maximal rank). 
  The $T$-fixed point set is $W(G)/W(K_0)$.  The Lie algebra of $T$ consists of diagonal blocks of the type
 $(0, A(\theta_1), A(\theta_2))$, where 
 $$A(\theta):=\left(
\begin{array}{cccccc} 
0  & -\theta \\
\theta & 0
\end{array}
\right). 
 $$
We regard the roots as elements of the Lie algebra $\mft$ of $T$, by putting on this a $T$-in\-variant metric.
Then $G$ and $K_0$ have a common root, namely $\alpha_1:=(0,A(1), A(0))$.
 The Weyl group $W(K_0)$ is generated by the reflection through the line perpendicular to $\alpha_1$.
Thus    $W(K_0)$ is the $W(G)$-stabilizer of the root $\alpha_1+\alpha_2$ in Figure \ref{fig1}
 (the left hand side diagram represents the roots of $G$, which lie in the Lie algebra of $T$).  That is, the
 set $(G/K_0)^T$ can be identified with the $W(G)$-orbit of $\alpha_1+\alpha_2$. This makes four points altogether, see the right hand side
diagram in Figure \ref{fig1}; we can also see there the six $T$-invariant 2-spheres in $G/K_0$, which are represented by thickened lines.   
 
 \begin{figure}[htb]
 \includegraphics[width=9.5cm]{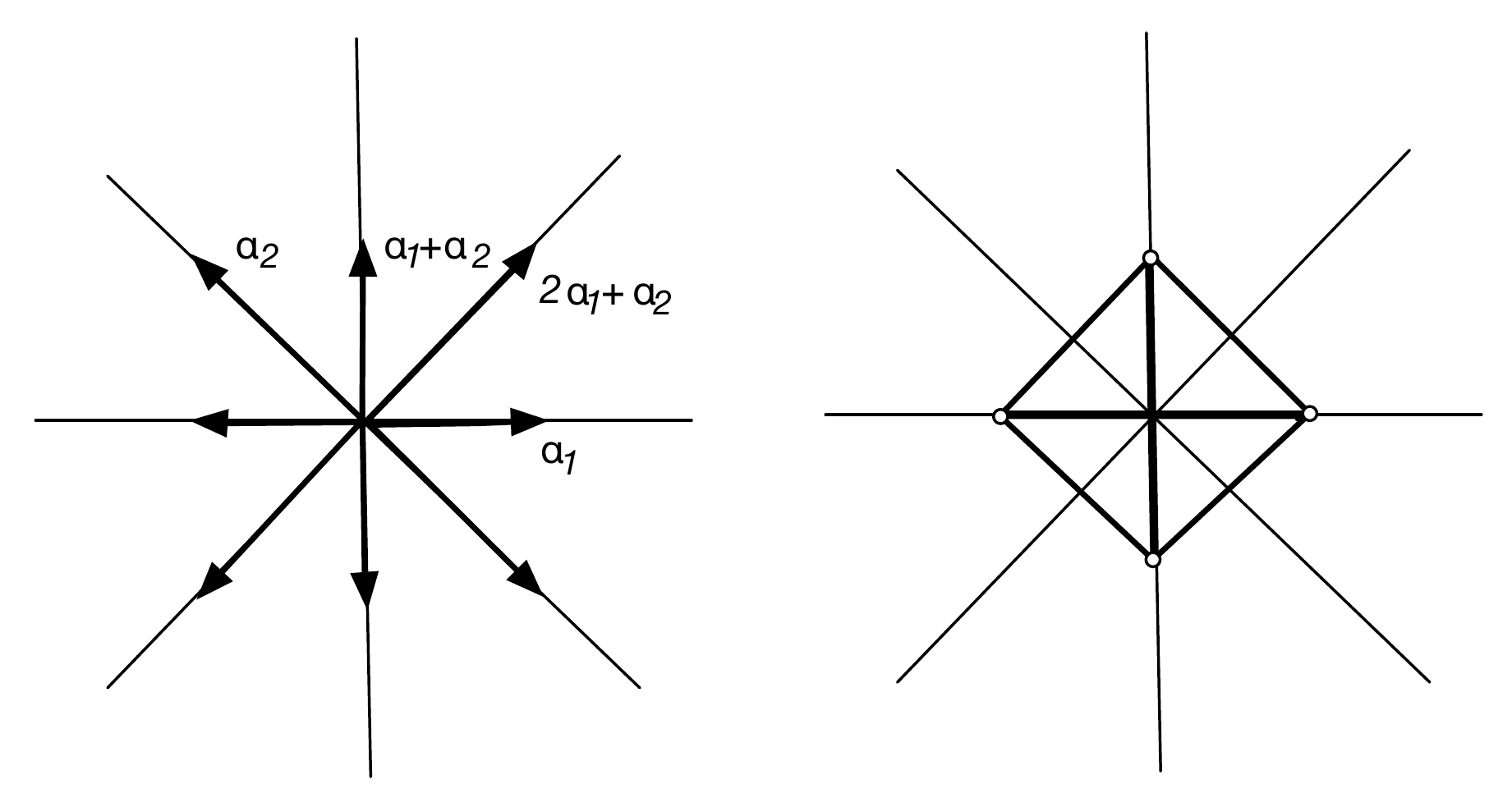}
 \caption{}
 \label{fig1}
 \end{figure}

 We would like now to describe the fixed points of the $T$-action on $G/K$. Since $\pi : G/K_0 \to G/K$ is a $\ZZ_2$-covering,
 which is also $T$-equivariant, we deduce that $(G/K)^T= \pi((G/K_0)^T)$: indeed, if
 $p=\pi(q)\in G/K$ is $T$-fixed, where $q\in G/K_0$, then the subspace $Tq$ of $G/K_0$ is contained in $\pi^{-1}(p)$,
 thus it must be equal to $\{q\}$. 
 
 In order to see concretely what $(G/K)^T$ is, one needs to understand the action of $g_0$ on $(G/K_0)^T$.
 Concretely, $g_0$ induces an element of the Weyl group of $G$ and we need to understand how does this act on
 the roots $\alpha_1$ and $\alpha_1+\alpha_2$ (see again Figure \ref{fig1}). We have  $\alpha_1=(0,A(1), A(0))$ and $\alpha_1+\alpha_2 =(0,A(0), A(1))$.
Thus
\begin{equation}\label{ad}{\rm Ad}_{g_0}(\alpha_1) = -\alpha_1, \quad {\rm Ad}_{g_0}(\alpha_1+\alpha_2) = -(\alpha_1 +\alpha_2).
\end{equation}
By identifying $G/K_0$ with the $G$-adjoint orbit of $\alpha_1+\alpha_2$, we conclude that $(G/K)^T$ has two fixed points, namely the cosets modulo $\ZZ_2$ of $\alpha_1$ and $\alpha_1+\alpha_2$.

The weights of the $T$-isotropy representation at any of these two fixed points are pairwise linearly independent, since 
the same is true for all isotropy representations of $G/K_0$ at the fixed points and $\pi: G/K_0 \to G/K$ is a $T$-equivariant local
diffeomorphism. We deduce that the 1-skeleton of $G/K$ consists of the fixed point sets of the same subtori of $T$ 
as those which determine the 1-skeleton of $G/K_0$.  Thus the 1-skeleton of $G/K$ consists of the images under $\pi$ of the six spheres represented in Figure \ref{fig1}.
By equation (\ref{ad}), the ``horizontal" and the ``vertical" spheres are invariant under ${\rm Ad}_{g_0}$: indeed, the latter is an element of the Weyl group and any such element permutes the $T$-invariant spheres. 
Hence the image of these two spheres are copies of $\RR P^2$.  For similar reasons, each of the two ``oblique" spheres on the
left hand side of the vertical axis are transformed under ${\rm Ad}_{g_0}$ into one of the other two ``oblique" spheres;
under $\pi$, out of the four ``oblique" spheres in $G/K_0$ one obtains only two in $G/K$. 
Let $T_1$ and $T_2$ be the circles (1-dimensional subtori in $T$) which leave these two spheres pointwise fixed.
Their Lie algebras are the kernels of $\alpha_2$, respectively $2\alpha_1+\alpha_2$. Now, a general element of
$T$ is a diagonal block of the form $(1, R(\theta_1), R(\theta_2))$, where    
 $$R(\theta):=\left(
\begin{array}{cccccc} 
\cos\theta  & -\sin \theta \\
\sin \theta & \cos\theta
\end{array}
\right). 
 $$
This gives rise to a natural presentation $T=S^1 \times S^1$ and also to the following expressions:
\begin{align*}
{}& T_1 =\{(z,z) \mid z\in S^1\}\\
{}& T_2= \{(z, z^{-1}) \mid z\in S^1\}.
\end{align*}
These are the labels carried by the two edges  in the GKM graph, see Figure \ref{fig2}.

 \begin{figure}[htb]
 \includegraphics[width=9.5cm]{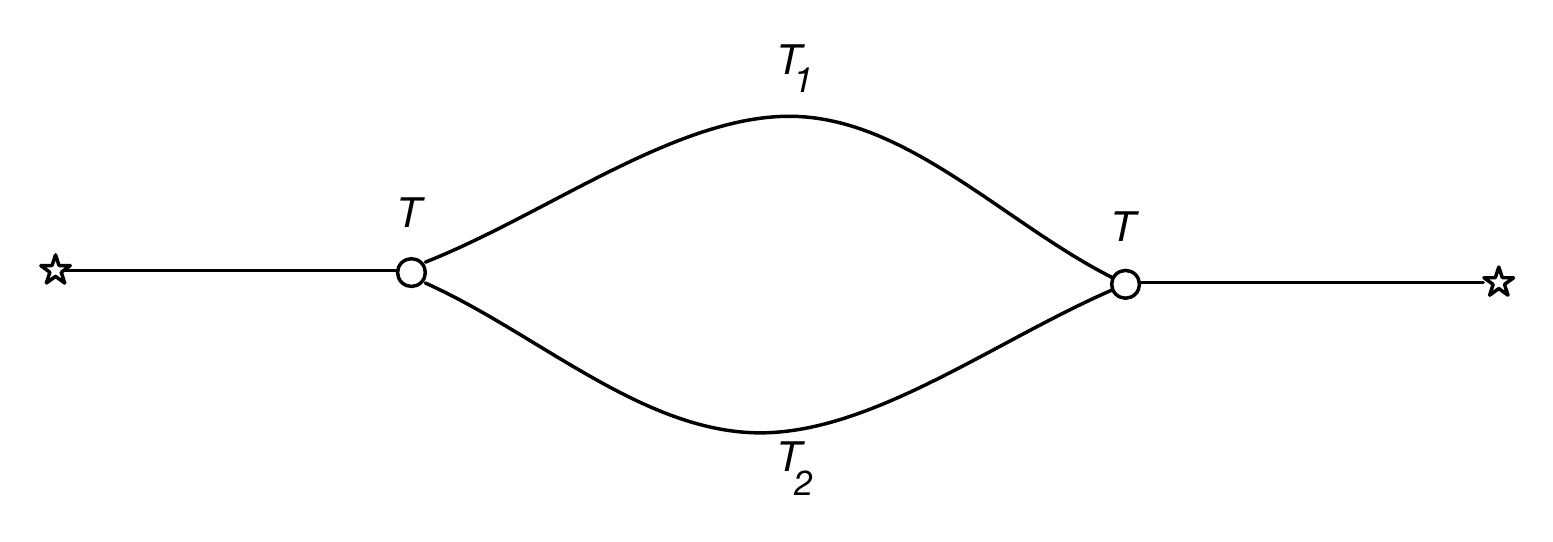}
 \caption{}
 \label{fig2}
\end{figure}

In conclusion, 
\begin{align*}
H^*_T(G_2(\RR^5))&\simeq \{(f, g)\in S(\mft^*)\oplus S(\mft^*)\mid f - g {\rm \ is \ divisible \ by \ } \alpha_2(2\alpha_1+\alpha_2)\}\\
&\simeq \{(f,g)\in \RR[s_1,s_2]^2\mid f(s,s)=g(s,s),\, f(s,-s)=g(s,-s), s\in \RR\}.
\end{align*}

 \end{ex}

\section{The non-abelian 1-skeleton}\label{sec:non-abeliansk}

Let us consider a group action $G\times M \to M$, where $G$ is a 
compact connected Lie group and $M$ a compact connected manifold.
Set $r:=\rk G$.
If the action is equivariantly formal, then, by Theorem \ref{thm:CS}, we can identify $H^*_G(M)$ 
with an $S(\mfg^*)^G$-subalgebra of $H^*_G(\Mmax)$, namely with the image of the map
$H^*_G(M_{r-1})\to H^*_G(\Mmax)$. 
Our next goal is to obtain a description of this image. This will be done under
the assumption that the following {\it non-abelian GKM conditions} are fulfilled:
\begin{enumerate}
\item[(i)] The $G$-action is equivariantly formal.
\item[(ii)] The space $\Mmax$ consists of isolated $G$-orbits.
\item[(iii)] For every $p\in \Mmax$, the weights of the isotropy representation of $G_p$ on $T_pM$ are pairwise linearly independent.
\end{enumerate}
These conditions will be in force throughout the rest of the paper.
In this section we obtain some preliminary results, which concern
exclusively the structure of $M_{r-1}$. (For obvious reasons, this space deserves
to be called the {\it non-abelian 1-skeleton} of the $G$-action.)  

\subsection{Abelian and non-abelian GKM conditions }
Let us fix a maximal torus $T\subset G$. Recall from Section \ref{section:one} the notation $M_i:=\{p\in M\mid \rk G_p \geq i\}$, and let us also define $M_{i, T}:=\{p\in M \mid \dim T_p \ge i \}$.

\begin{lem}\label{lem:TandGstrata} $M_{i} = G M_{i, T}$.
\end{lem}

\begin{proof} If $p\in M_{i}$, then $G_p$ contains a maximal torus $T'$  of dimension at least $i$.
There exists $g\in G$ such that $gT'g^{-1} \subset T$. Consequently $T_{gp}=G_{gp}\cap T =
gG_pg^{-1} \cap T $ contains $gT'g^{-1}$, and thus has dimension at least $i$.
Conversely, if $p\in M_{i, T}$, then $G_p$ contains $T_p$, hence it has rank at least $i$.
\end{proof}

\begin{lem}\label{ift} The $G$-action satisfies the non-abelian GKM conditions if and only if  $M$  is a GKM space relative to the  $T$-action (see Section \ref{abelian}).
\end{lem}
\begin{proof} Equivariant formality of the $G$- and the $T$-action are equivalent by \cite[Proposition C.26]{GGK}. By Lemma \ref{lem:TandGstrata}, we have $\Mmax = GM^T$, so if $M^T$ is finite, then $\Mmax$ consists of finitely many $G$-orbits. Conversely, if $\Mmax$ consists of finitely many $G$-orbits, then $M^T$ is finite because there are only finitely many $T$-fixed points in a homogeneous space of the form $G/H$ with $\rk G = \rk H$ (more precisely, the fixed point set is the quotient of Weyl groups $W(G)/W(H)$). The conditions on the isotropy representations are obviously equivalent because the weights of the $G_p$-representation are by definition the weights of the representation of the maximal torus $T\subset G_p$.
\end{proof}

\begin{ex}\label{examp} 
 If $G$ is a semisimple compact Lie group, and $H\subset G$ a  connected closed subgroup with $\rk H = \rk G$, then the $H$-action on $G/H$ by left multiplication satisfies the non-abelian GKM conditions. In fact, it is shown in \cite[Theorem 1.1]{GHZ} that the action of a maximal torus in $H$ satisfies the GKM conditions. 
 The special cases  $\Sp(n)/\Sp(1)^n$ and $F_4/\Spin(8)$
 are investigated in \cite{M} and \cite{MW} respectively.  
 \end{ex}

\begin{lem}\label{lem:gkmcond} If the $G$-action satisfies the non-abelian GKM conditions, then
the induced action of $T$ on $\Mmax$ satisfies  the usual GKM conditions,
see (i)-(iii) in Section \ref{abelian}.
\end{lem}
\begin{proof}
Any component of $\Mmax$ is   a homogeneous space of the form $G/H$, where $H$ has maximal rank in $G$.  Since $H^{\rm odd}(G/H)=0$, the $T$-action on this component is equivariantly formal. Conditions (ii) and (iii) in the definition of 
GKM space, see Section \ref{abelian}, are also clearly satisfied.   \end{proof}

\begin{ex}
For equivariantly formal torus actions, Lemma \ref{fixp} shows that $\Mmax=M^T$ is the infinitesimal bottom stratum for the action. For general $G$-actions this is no longer true, even if they satisfy the non-abelian GKM conditions. As an example, consider the cohomogeneity-one action of $G=S^3\times S^3$ with group diagram $S^3\times S^3\supset S^1\times S^1, \Delta(S^3)\supset \Delta(S^1)$, where $\Delta(H)$ denotes the diagonally embedded subgroup $\{(h,h)\mid h\in H\}$ in $S^3\times S^3$.
One of the singular isotropies is of maximal rank, thus $\Mmax=S^3\times S^3/S^1\times S^1$. Furthermore, this implies using \cite[Corollary 1.3]{GM} that the action is equivariantly formal. The second singular orbit is part of the infinitesimal bottom stratum, but not a subset of $\Mmax$. Condition (iii) in the definition of a GKM space is also fulfilled: Let $p\in \Mmax$ with  $G_p=S^1\times S^1$. The three weights of the isotropy representation at $p$ are pairwise linearly independent, as their kernels are given by $\{(x,0)\}\subset \mfg_p$, $\{(0,x)\}\subset \mfg_p$, respectively $\{(x,x)\}\subset \mfg_p$.
\end{ex}

\subsection{Abelian and non-abelian 1-skeletons}
Recall from Section \ref{abelian} that the GKM conditions for the $T$-action on $M$ imply that $M_{r-1,T}$ can be written as a finite union 
$\bigcup_{i\in I}S_i\cup\bigcup_{j\in J} P_j $
where $S_i$
are 2-spheres and $P_j$ real projective planes.  Each $S_i$ and each  $P_j$ is a connected component of the fixed point set of a 
codimension-one torus $T_i\subset T$, resp.~$T_j\subset T$; $S_i$ contains exactly two $T$-fixed points, say $p_i$ and $q_i$; $P_j$ contains one $T$-fixed point, say $r_j$. Observe that 
$M_{r-1}=\bigcup_{i\in I} GS_i \cup \bigcup_{j\in J}GP_j $ by Lemma \ref{lem:TandGstrata}.
 Each space $GS_i$ contains the components $Gp_i$ and $Gq_i$ of $\Mmax$; similarly, $GP_j$ contains $Gr_j$. The next lemma shows that they contain no other components of $\Mmax$; note, however, that the two components of $GS_i\cap \Mmax$ might coincide, see  Example \ref{ex:hp1} below.

\begin{lem}\label{lem:take} \begin{enumerate}
\item[(a)] Take $i\in I$ such that $GS_i\neq Gp_i$, i.e., $GS_i$ is not just a $G$-orbit.
Then $S_i\cap \Mmax =\{p_i, q_i\}$ and consequently 
 $GS_i\cap \Mmax=Gp_i\cup Gq_i$.
\item[(b)] Take $j\in J$ such that $GP_j\neq Gr_j$, i.e., $GP_j$ is not just a $G$-orbit.
Then $P_j\cap \Mmax =\{r_j\}$ and consequently 
 $GP_j\cap \Mmax=Gr_j$.
\end{enumerate}
\end{lem}

\begin{proof}  (a) Take $p\in S_i\cap \Mmax$. There exists $g\in G$ such that
$gTg^{-1}$ fixes $p$, hence $gp\in M^T$. 
The orbit $Ggp=Gp$ has a common point with $S_i$ at $p$. 
The $T$-action on this orbit satisfies the usual GKM conditions, see Lemma  
\ref{lem:gkmcond}. There are two cases:
first, if $p$ is an isolated fixed point of $T_i$ in $Gp$, then $p$ is fixed by $T$, hence 
$p\in \{p_i, q_i\}$; 
second, if $p$ is not isolated in $(Gp)^{T_i}$, then the positive dimensional connected components 
of the latter set are 2-spheres or real projective planes, 
and one of them contains $p$. On the other hand, $S_i$ is the connected component of $M^{T_i}$ which contains $p$. Hence $S_i\subset Gp$, which implies that $GS_i=Gp$,
contradiction.

(b) We use the same argument as in the proof of (a).
\end{proof}

\begin{ex}\label{ex:hp1} The  orbits $Gp_i$ and $Gq_i$ mentioned in point (a) of the lemma are not necessarily
different. Let us consider for example the quaternionic projective line $\HH P^1$, which consists
of all $1$-dimensional $\HH$-linear subspaces of $\HH^2$. It has a canonical transitive action of
the symplectic group $\Sp(2)$ of all $2\times 2$ matrices $A$ with entries in
 $\HH$ satisfying $AA^*=I_2$.  We are  interested in the action of the group 
which consists of all elements in $\Sp(2)$ whose entries are in $\CC$, where the latter is canonically 
embedded in $\HH$: this group is obviously isomorphic to the unitary group $\U(2)$.
That is, we look at the cohomogeneity-one action of $G:=\U(2)$ on $M:=\HH P^1$. We choose the standard maximal torus 
$T$ in $\U(2)$, i.e., the space of all diagonal unitary matrices. The $T$-fixed points are $[1:0]$
and $[0:1]$ (see, e.g., \cite{M}). As $[0:1]=g[1:0]$ for $g=\left(\begin{matrix}0 & 1 \\ 1 & 0\end{matrix}\right)\in \U(2)$, 
there is just one orbit of the type $Gp$, $p\in M^T$. See Example \ref{u2} for a continuation of this example. \end{ex}

The intersection of two submanifolds of the form $S_i$ or $P_j$
 consists of at most two points in $M^T$. We are now interested in the intersection of the  $G$-orbits of two such spaces:

\begin{lem}\label{lem:intersection}
\begin{enumerate}
\item[(a)] If $i_1, i_2 \in I$ are such that  $GS_{i_1}\neq GS_{i_2}$ then the intersection $GS_{i_1} \cap GS_{i_2}$ is empty or consists of
one or two $G$-orbits in $\Mmax$. 
\item[(b)] If $j_1, j_2 \in J$ are such that  $GP_{j_1}\neq GP_{j_2}$ then the intersection $GP_{j_1} \cap GP_{j_2}$ is empty or consists of
one  $G$-orbit in $\Mmax$. 
\item[(c)] If $i\in I$ and $j \in J$ are such that  $GS_i\neq GP_j$ then the intersection $GS_{i} \cap GP_{j}$ is empty or consists of
one  $G$-orbit in $\Mmax$.
\end{enumerate}
\end{lem}

\begin{proof} 
We only give the details of the proof for (a); the other two points can be proven by similar methods.  Because of Lemma \ref{lem:take}, we only need to show that $GS_{i_1}\cap GS_{i_2}$ is contained in $\Mmax$. 
Indeed, assume that there exists $p\in S_{i_1}$, $p\notin \Mmax$
such that $gp\in S_{i_2}$ for some $g\in G$. We have $T_p=T_{i_1}$ and $T_{gp}=T_{i_2}$, hence 
$T_{i_2}=T_{gp}= gT_pg^{-1}\cap T =gT_{i_1}g^{-1}\cap T\subset gT_{i_1}g^{-1}$,
which actually means $T_{i_2} =gT_{i_1}g^{-1}$. This implies that $M^{T_{i_1}}=g^{-1}M^{T_{i_2}}$.
The connected components of $p$ in the previous two spaces are
$S_{i_1}$, respectively $g^{-1}S_{i_2}$, hence $S_{i_1} = g^{-1}S_{i_2}$ and consequently
$GS_{i_1} = GS_{i_2}$, contradiction.
 \end{proof}

\begin{ex} It can happen already for torus actions that the intersection $GS_{i_1} \cap GS_{i_2}$ consists of two orbits (i.e., fixed points in this case). A simple example is given by the restriction of the action in Example \ref{ex:hp1} to a maximal torus of $\U(2)$.
\end{ex}

\subsection{$G$-orbits along $T$-transversal geodesics}
The results of this section are  important tools 
in studying the structure of $M_{r-1}$. They require the choice of a  
$G$-invariant Riemannian metric on $M$. We fix  $i\in I$ and observe that the sphere $S_i$ is $T$-invariant.
We consider an injective geodesic segment $\mu _i :[0,1]\to S_i$ which is $T$-transversal, i.e., perpendicular to each $T$-orbit it intersects,
and such that  $\mu _i(0)=p_i$ and $\mu _i(1)=q_i$. Since $S_i=T\mu _i([0,1])$, we have $GS_i =G\mu _i([0,1])$. We denote by $G_i$ the isotropy group of $\mu _i$, i.e., the group of elements that fix $\mu _i$ pointwise. Clearly, $G_i\subset G_{\mu _i(t)}$ for all $t$.

We assume that $GS_i\neq Gp_i$. Recall that $S_i$ is a component of the
fixed point set of a certain codimension-one subtorus $T_i\subset T$.
By Lemma \ref{lem:take}, the torus $T_i$ is maximal in both $G_i$ and $G_{\mu _i(t)}$ for all $0<t<1$; we may therefore form the Weyl groups $W(G_i)$ and  $W(G_{\mu _i(t)})$ (for $0<t<1$) with respect to $T_i$. As $W(G_i)\subset W(G_{\mu_i(t)})$ for such $t$, we can consider
\[
H^*_G(G\mu _i(t)) = H^*(BG_{\mu _i(t)})=H^*(BT_i)^{W(G_{\mu _i(t)})}
\]
as a subring of
\[ H^*(BG_i)=H^*(BT_i)^{W(G_i)}.
\]

\begin{lem}\label{gsi} Assume that $GS_i\neq Gp_i$. Then one of the following two cases holds:
\begin{enumerate}
\item[(a)] For all $0<t<1$, we have $G\mu _i(t)\cap S_i = T\mu _i(t)$. In this case, $W(G_{\mu _i(t)})=W(G_i)$ and $H^*(BG_{\mu _i(t)})=H^*(BG_i)$.
\item[(b)] For all $0<t<1$, we have $G\mu _i(t)\cap S_i = T\mu _i(t)\cup T\mu _i(1-t)$. Then there exists $g_i\in W(G_{\mu _i(\frac12)})\setminus W(G_{i})$ such that $g_i\mu _i(t)=\mu _i(1-t)$. In particular, $Gp_i=Gq_i$. For $0<t<1$, $t\neq \frac12$ we have $ W(G_{\mu _i(t)})=W(G_i)$ and $ H^*(BG_{\mu _i(t)})=H^*(BG_i)$.
\end{enumerate}
\end{lem}
\begin{proof}
Assume that there exists $0<t_0<1$ such that $G\mu _i(t_0)\cap S_i$ is strictly larger than $T\mu _i(t_0)$, i.e., we are not in case (a). Then there exists $h\in G$ such that $h\mu _i(t_0)=\mu _i(t_1)$ for some $t_1\neq t_0$, $t_1\in (0,1)$. We may assume $t_0<t_1$. Both $T_i$ and $hT_ih^{-1}$ are maximal tori in $G_{\mu _i(t_1)} = hG_{\mu _i(t_0)}h^{-1}$, so there exists $k\in G_{\mu _i(t_1)}$ such that $khT_i(kh)^{-1}=T_i$. Define $g_i:=kh$. Because $g_i$ normalizes $T_i$ and sends $\mu _i(t_0)$ to $\mu _i(t_1)$, it maps $S_i$ to itself. Therefore, the curve $g_i\mu _i$ is contained in $S_i$ and is (because $g_i$ acts as an isometry) perpendicular to the $T$-orbits in $S_i$. Thus, there are only two possibilities: either $g_i\mu _i(t) = \mu _i(t+t_1-t_0)$ for all $t$ or $g_i\mu _i(t)=\mu _i(t_0+t_1-t)$ for all $t$. The first case is impossible because $\mu _i(t_1-t_0)$ is not contained in $\Mmax$. Therefore, $g_i\mu _i(t)=\mu _i(t_0+t_1-t)$. This implies that $\mu _i(t_0+t_1)=q_i$, so $t_0+t_1=1$, i.e., $g_i\mu _i(t)=\mu _i(1-t)$ for all $t$; thus, the first statement in case (b) is true.

It remains to show that in case (a), $W(G_i)=W(G_{\mu _i(t)})$ for $0<t<1$ and the same equality holds in case (b) for $0<t<1$, $t \neq \frac12$. But any element in 
$W(G_{\mu _i(t)})$ that does not fix $\mu _i$ pointwise has to reflect the geodesic $\mu_i$ in the point $\mu_i(t)$ by the same arguments as above, which is only possible for $t=\frac12$ in case (b).
\end{proof}

Let us now take  $j\in J$ and consider the real projective plane $P_j$, which is
a component of $M^{T_j}$, $T_j$ being a certain codimension-one subtorus 
of $T$. We also consider an injective geodesic segment $\eta  _j :[0,1]\to P_j$ which is $T$-transversal and  such that 
$\eta  _j(0)=r_j$ and $T\eta  _j(1)$ is an exceptional $T$-orbit. Then $\eta  _j([0,1])$ intersects each orbit of the $T$-action on $P_j$ exactly once. We denote by $G_j$ the subgroup of $G$ which fixes $\eta  _j$ pointwise.

We assume now that $GP_j\neq Gr_j$. By Lemma \ref{lem:take}, the torus $T_j$ is maximal in both $G_j$ and $G_{\eta  _j(t)}$ for $0<t\leq 1$. For such $t$, we have $W(G_j)\subset W(G_{\eta  _j(t)})$, and therefore
\[
H^*_G(G\eta  _j(t))=H^*(BG_{\eta  _j(t)}) = H^*(BT_j)^{W(G_{\eta  _j(t)})}
\]
is a subring of 
\[
H^*(BG_j)=H^*(BT_j)^{W(G_j)}.
\]
With the methods of the previous lemma, we can show as follows:

\begin{lem}\label{gsii} Assume that $GP_j\neq Gr_j$. Then for all  $0<t<1$ we have $G\eta  _j(t)\cap P_j =T\eta  _j(t)$, $W(G_{\eta  _j(t)})=W(G_j)$, and $ H^*(BG_{\eta  _j(t)})=H^*(BG_j)$.
\end{lem}

 \begin{rem}\label{rem:compnotsmooth} Lemmata \ref{gsi} and \ref{gsii} show that if the spaces $GS_i$ or $GP_j$ are smooth (and not a single $G$-orbit),
 then they become cohomogeneity-one manifolds relative to the $G$-action. There are situations 
 when this is not the case. Let us consider for instance the canonical action of $H$ on $K/H$, where
 the latter is an irreducible symmetric space of rank at least two and satisfies $\rk K = \rk H$ (take for instance the Grassmannian $\SU(4)/{\rm S}(\U(2)\times \U(2))$).     
 This action satisfies the non-abelian GKM hypotheses, by \cite{GHZ} and Lemma \ref{ift} above. Let $T\subset H$ be a maximal torus.
 The coset $eH$ is a fixed point of the $T$-action and we consider
 a two-dimensional submanifold $R$ in the 1-skeleton of the $T$-action which contains it. We claim that the space $HR$ is not  smooth.
 Indeed, if it were, then the $H$-action on it would have cohomogeneity-one with orbit space $[0,1]$ and
 $eH$ would be a singular orbit. From the slice theorem, $H$ then acts transitively on the unit sphere in $T_{eH}HR$.
 Thus $T_{eH}HR$ is an irreducible $H$-submodule of $T_{eH}K/H$. Since the symmetric space is irreducible,
 we deduce that $T_{eH}HR=T_{eH}K/H$ and conclude that $H$ acts transitively on the unit sphere in
 $T_{eH}K/H$: this is in contradiction with the cohomogeneity of the isotropy representation of
 $H$ on $K/H$ at $eH$ being $\rk K/H$, which was assumed to be at least two.
 
 Another concrete situation when $GS_i$ is not smooth will be presented
in Examples \ref{typec} and \ref{typecc}.
 \end{rem}

\section{Non-abelian GKM theory}\label{sec:non-abelian}

In this section we accomplish the goal mentioned at the beginning of the previous section, that is, we identify the image of the (injective) map $H^*_G(M_{r-1})
\to H^*_G(\Mmax)$ in the presence of the non-abelian GKM conditions.
The main results are Propositions \ref{lem:a} and \ref{prop:b};
for a more systematic presentation, see Section 
\ref{sec:non-abelianGKMgraphs},  especially Theorem  \ref{via}.

We would like to emphasize that, in contrast with the previous section, 
all considerations made here will be independent of the choice of a maximal
torus $T$. 
The first step will be an intrinsic description   of the non-abelian 1-skeleton $M_{r-1}$,
which is going to be presented in the following subsection.

\subsection{The components of the 1-skeleton} 
Consider the connected components of
$M_{r-1}\setminus \Mmax$ and their  closures: these will be called
the {\it components of the 1-skeleton}. 
Note that for any maximal torus $T\subset G$, these are nothing but the 
associated spaces $GS_i$ and $GP_j$ from the previous section which are not 
single $G$-orbits. We summarize some of the results concerning them which were obtained  in the previous section:
 
 \begin{itemize}
 \item Their union is $M_{r-1}$.
 \item Each of them is $G$-invariant and contains one or two $G$-orbits in $\Mmax$.
 \item The intersection of any two of them is contained in $\Mmax$ (possibly empty).
\item The $G$-orbit space of  any of them  is homeomorphic to a closed interval or a circle.
 \end{itemize}

The last statement follows from Lemmata \ref{gsi} and \ref{gsii}.
In fact, if $N$ is a component, then it fits into one of the following three
cases, see also
items (I)-(III) in the introduction:

\begin{itemize}
\item[(I)] $N/G\simeq [0,1]$ and $N$ contains exactly two orbits in $\Mmax$,
which correspond to $0$ and $1$.
\item[(II)] $N/G\simeq [0,1]$ and $N$ contains exactly one orbit in $\Mmax$,
which corresponds to $0$.
\item[(III)] $N/G\simeq \RR/\ZZ=S^1$ and $N$ contains exactly one orbit in $\Mmax$.
\end{itemize}

By fixing a maximal torus $T\subset G$ like in the previous section, the
components are of the form $GS_i$, $i\in I$ or $GP_j$, $j\in J$.
More precisely, if $S_i$ is like in Lemma \ref{gsi} (a) and $Gp_i \neq Gq_i$,
then $GS_i$ is of type (I), whereas if $Gp_i = Gq_i$, then 
$GS_i$ is of type (III); if $S_i$ is like   in Lemma \ref{gsi} (b),
then $GS_i$ is of type (II). Also of type (II) is any component of the form $GP_j$
like in Lemma \ref{gsii}. 

We can illustrate these situations by concrete examples.
Of type (I) are: any 2-sphere in the 1-skeleton of a GKM torus action
as well as the spaces $GS_1$ and $GS_2$ to be mentioned in  
Example \ref{sp22} below. 
In Example \ref{ex:gras}, the two $\RR P^2$'s in the 1-skeleton are
of type (II); of the same type is the component $GS$ in Example \ref{ex:hp1} (and  \ref{u2} below).
We will now give an example of a component of type (III).
(Note that, by Remark \ref{rem:compnotsmooth} and the general theory of cohomogeneity-one actions, a type (III) component is not smooth.)

\begin{ex}\label{typec}
 Consider the compact, connected, and simply connected  Lie group of type $G_2$ along with $M$, the adjoint orbit
of the point $p$ in Figure \ref{fig3}. That diagram  represents the roots of our  group,
which lie in the Lie algebra $\mft$ of a maximal torus  $T$.
 \begin{figure}[htb]
 \includegraphics[width=7.5cm]{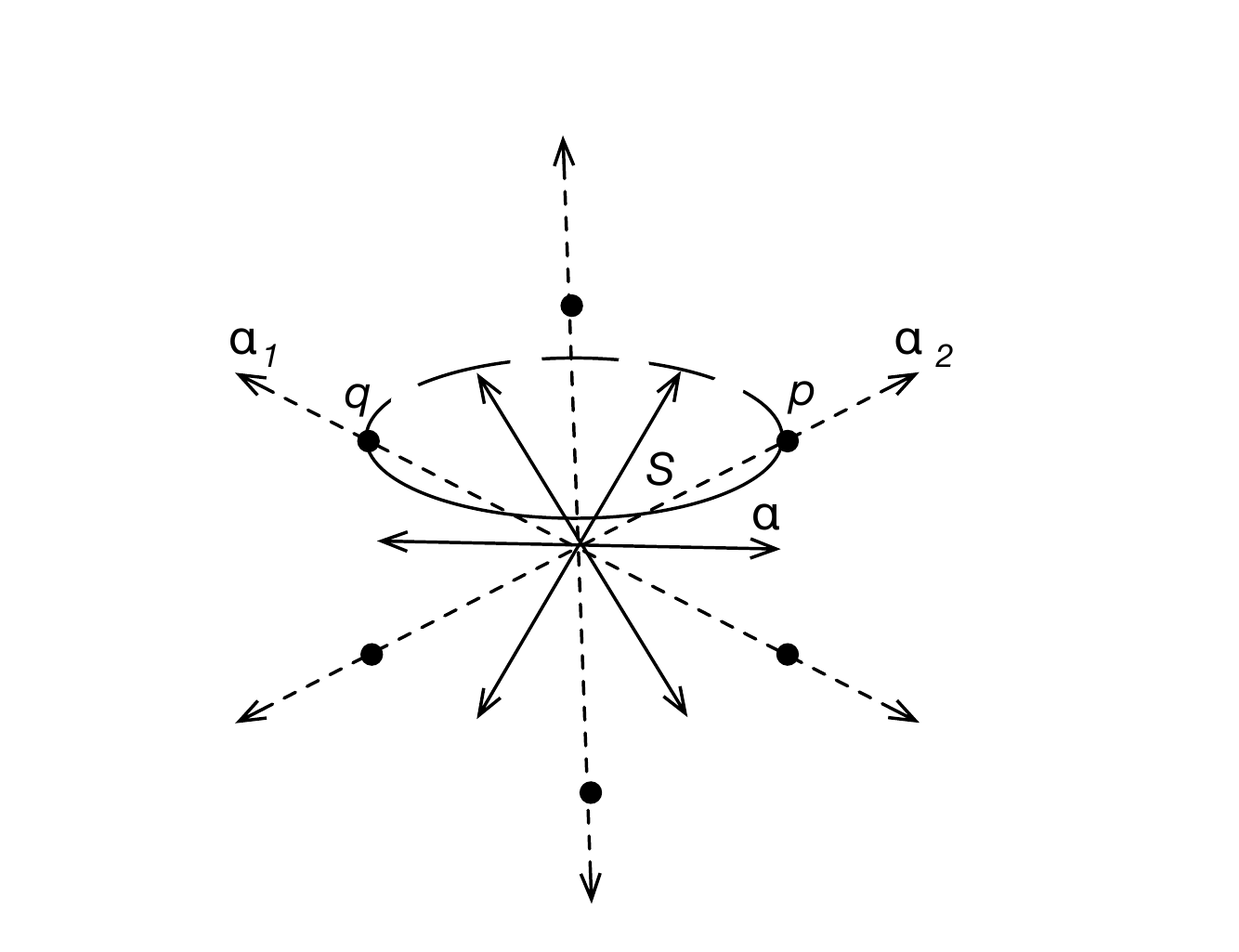}
 \caption{}
 \label{fig3}
\end{figure}
Denote by $K$
the connected subgroup which contains $T$ and whose roots are the vectors
represented by dotted lines in the diagram. Note that $K$ is isomorphic to $\SU(3)$.
The relevant action for us is the one of $K$ on $M$.
It satisfies the GKM conditions, see Example \ref{examp}.
The $T$-fixed points of $M$ are the six dots in the diagram: they turn out to be
 the orbit of $p$ under the Weyl group $W(K)$.
Especially important for us are $p$ and $q:=s_\alpha(p)$
(in general, if $\gamma$ is a root, then $s_\gamma$ is the
reflection in the line $\gamma^\perp$).
There is a  $T$-invariant 2-sphere in $M$ between $p$ and $q$, which we denote by 
$S$. It is not contained in $Kp$, because $Kp$ is a $K$-adjoint orbit
and $\alpha$ is not a root of $K$. But $p$ and $q$ are in the same $W(K)$-orbit,
hence we have $Kp=Kq$. We will now show that $S$ satisfies the
conditions in Lemma \ref{gsi} (a): this implies immediately that $KS/K\simeq S^1$.  
Concretely, if $\mu: [0,1]\to M$ is a $T$-transversal injective geodesic segment
 in $S$ with $\mu(0)=p$ and $\mu(1)=q$, then
 $K\mu(t) \cap S =T\mu(t)$, for all $0<t<1$.
Indeed, let us assume that the contrary is true. Like in the proof of Lemma \ref{gsi}, this implies that there exists
$k \in K$ such that $k \mu(t)=\mu(1-t)$ for all $t\in (0,1)$.
In particular, $kp=q$ and $kq=p$.
Since $p\in \mft $ is a regular point for the adjoint action of 
$K$,
 the $K$-stabilizer of $p$ is equal to $T$.
Hence $kp=q$ implies readily that $k\in N_K(T)$.
Finally observe, by inspection, that for
any of the six elements $w$ of  $W(K)$, we have either $wp\neq q$ or $wq\neq p$. See Example \ref{typecc} for a continuation of this example.
\end{ex}

The following notion will play an important role:

\begin{dfn} Let $N$ be a component of the non-abelian
1-skeleton. An injective geodesic segment $\gamma : [0,1] \to M$ is {\it adapted} to
$N$ if:
\begin{itemize}
\item the trace of $\gamma$ is contained in $N$ 
\item $\gamma(0) \in \Mmax$
\item $\gamma$ is $G$-transversal, i.e., it is perpendicular to any $G$-orbit it meets at the intersection point
\item if $N$ is of type (I) or (II), then   the map $[0,1]\to N/G$ given by $t\mapsto G\gamma(t)$ is a bijection;
if $N$ is of type (III), then the map $[0,1]/_{0\sim 1}\cong S^1 \to N/G$ given by $t\mapsto G\gamma(t)$
is well defined and bijective.
\end{itemize}
\end{dfn}

Assume that $T\subset G$ is fixed and $N$ is of the form $GS_i$ or $GP_j$  like in the previous section. In the situations described by Lemma \ref{gsi} (a), 
respectively Lemma \ref{gsii}, 
 $\mu_i$ is adapted to $GS_i$ and $\eta_j$ is adapted to $GP_j$; for $S_i$  like in Lemma \ref{gsi} (b), a geodesic segment adapted to $GS_i$ is
$s\mapsto \mu_i(s/2)$, $0\le s\le 1$. Any such curve remains an adapted geodesic segment after translations by elements of $G$ or, if $N$ is of type (I), the reparametrization induced by $s\mapsto 1-s$, $s\in [0,1]$. The following   converse
result  holds:

\begin{lem}\label{lem:adapted} Let $N$ be a component of the non-abelian 1-skeleton. Then there exists a geodesic segment adapted to $N$; it is unique up to $G$-translations and, if $N$ is of type (I),
 the reparametrization induced by $s\mapsto 1-s$, $s\in [0,1]$.
\end{lem}

\begin{proof} 
Let $p\in N$ be any point with $\rk G_p=r-1$. Let $K\subset G_p$ be a maximal torus and consider the component $R$ of $M^K$ containing $p$.
Let also  $T$ be a maximal torus in $G$ with $K\subset T$.
The point $p$ is not in $M^T$, hence, by Lemma \ref{ift},  $R$  is a two-dimensional
$T$-invariant submanifold of $M$. By Lemma \ref{fixp} the $T$-action on
$R$ fits into one of the cases described in Lemma  \ref{lett}.
Consequently we have $T_pTp\subset T_pR$, hence $(T_pGp)^\perp \cap T_pR$
has dimension at most 1. The tangent vector at $p$ to any admissible geodesic segment which goes through $p$  is contained, up to a $G_p$-translation, in this intersection. Indeed, if $v\in T_pM$ is such a vector, then the group $(G_p)_v$  (being equal to the $G$-stabilizer of points on the geodesic that are sufficiently close
to $p$) has rank $r-1$ and hence contains a $G_p$-conjugate of $K$:
thus, the geodesic segment is, up to a $G_p$-translation, contained in $R$.
 The claim follows because adapted geodesic segments are supposed to meet every $G$-orbit in $N$.    \end{proof}

For each component $N$ we choose a geodesic $\gamma_N$ adapted to it, along with
the group $G_{\gamma_N}$, the pointwise $G$-stabilizer of $\gamma_N$.
Both stabilizers
 $G_{\gamma_N(0)}$ and $G_{\gamma_N(1)}$   contain $G_{\gamma_N}$. 
Let $\mfg_{\gamma_N}$, $\mfg_{\gamma_N(0)}$ and $\mfg_{\gamma_N(1)}$ 
 be the Lie algebras of these groups.
We will make the following identifications:
\begin{align*} {}& H^*_G(G\gamma_N(0))=H^*(BG_{\gamma_N(0)})=
S(\mfg_{\gamma_N(0)}^*)^{G_{\gamma_N(0)}},\\
{}& H^*_G(G\gamma_N(1))=H^*(BG_{\gamma_N(1)})=
S(\mfg_{\gamma_N(1)}^*)^{G_{\gamma_N(1)}},\\ 
 {}& H^*_G(G\gamma_N(t))=H^*(BG_{\gamma_N})=S(\mfg_{\gamma_N}^*)^{G_{\gamma_N}},\end{align*}
 where $0<t<1$. 
In this way, the canonical  maps $H^*(BG_{\gamma_N(0)})\to H^*(BG_{\gamma_N})$ and $H^*(BG_{\gamma_N(1)})\to H^*(BG_{\gamma_N})$ are those induced 
by restriction from $\mfg_{\gamma_N(0)}$ respectively $\mfg_{\gamma_N(1)}$ to $\mfg_{\gamma_N}$.

\subsection{The 1-skeleton and equivariant cohomology}
In Propositions \ref{lem:a} and \ref{prop:b} we will give a complete description of the $S(\mfg^*)^G$-algebra $H^*_G(M)$. First we need a Lemma, whose proof is essentially the same as that of \cite[Proposition 4.2]{GR}\footnote{Note that there is a slight error in \cite[Proposition 4.2]{GR}: on the right hand side of the equation one has to consider the subgroup of $N_G(\mft_p)$ consisting of those elements which leave invariant $M^{\mft_p,p}$ instead of the whole normalizer $N_G(\mft_p)$.}.
\begin{lem}\label{lem:eqcohomgeodpiece} Let $N$ be a component of the nonabelian 1-skeleton, and $(a,b)\subset [0,1]$ be any interval. Let $T_N$ be a maximal torus in $G_{\gamma_N}$ (of rank $r-1$). Then
\[
H^*_G(G \gamma_N((a,b))) \simeq S(\mft_N^*)^{W(G_{\gamma_N})}
\]
as graded rings. 
\end{lem}  
\begin{proof} Let $T$ be a maximal torus in $G$ such that $T_N\subset T$.
We first calculate the fibers of the natural map
\begin{equation}\label{eq:lem1skeleton}
G\times_K (T \gamma_N((a,b))) \to G \gamma_N((a,b)),
\end{equation}
where $K=\{g \in N_G(T_N)\mid g (T \gamma_N((a,b))) \subset T \gamma_N((a,b))\}$. 
Assume that $[g,t\gamma_N(s_0)]$ is mapped to $\gamma_N(s_1)\in \gamma_N((a,b))$, i.e., $gt\gamma_N(s_0)=\gamma_N(s_1)$. By the last condition in the definition of an adapted geodesic segment, this implies that $s_0=s_1$. Thus, the fiber of $\gamma_N(s_1)$ is the set of elements of the form $[h,\gamma_N(s_1)]$ with $h\in G_{\gamma_N(s_1)}$, which is isomorphic to $G_{\gamma_N(s_1)}/G_{\gamma_N(s_1)}\cap K$. We clearly have $G_{\gamma_N(s_1)}\cap K\subset N_{G_{\gamma_N(s_1)}}(T_N)$. 
The converse inclusion holds, because the
connected component of $\gamma_N(s_1)$ in $M^{T_N}$
is in fact  a two-dimensional component of the abelian 1-skeleton of the $T$-action
on $M$ and is left invariant by any $g\in G$
which fixes $\gamma_N(s_1)$ and normalizes $T_N$;
the action of $T$ on this component is of one of the two types described in Lemma
\ref{lett} and both $\gamma_N$ and $g\gamma_N$ are contained in the component and
$T$-transversal relative to this action, which implies readily that $g\in G_{\gamma_N}$ and
$g (T \gamma_N(s))\subset T \gamma_N(s)$, for all $s\in (a,b)$.
 Thus, the fiber of $\gamma_N(s_1)$ is the quotient $G_{\gamma_N(s_1)}/N_{G_{\gamma_N(s_1)}}(T_N)$, which is acyclic by \cite[Lemma 3.2]{GR} (if $G_{\gamma_N(s_1)}$ is connected, see \cite[Sect.~III.1, Lemma (1.1)]{Hs}). The Leray spectral sequence of the Borel construction associated to \eqref{eq:lem1skeleton} thus collapses at the $E_2$-term, which implies that 
\[
H^*_G(G \gamma_N((a,b))) = H^*_K(T \gamma_N((a,b))).
\]
Next we note that $K$ is the subgroup of $G$ generated by $T$ and $N_{G_{\gamma_N}}(T_N)$ (given $k\in K$, we have $k \gamma_N(s_1)
\in T \gamma_N(s_1)$, hence there exists $t\in T$ such that 
$tk$ is in $G_{\gamma_N(s_1)}$ and normalizes $T_N$, which implies that
$tk\in G_{\gamma_N}$, by an argument already used  before). Since 
the identity component of  $N_{G_{\gamma_N}}(T_N)$ is $T_N$, 
the identity component of $K$ is $T$.
It follows that
\[
H^*_K(T \gamma_N((a,b))) = (S(\mft_N^*)\otimes H^*_{T/T_N}(T \gamma_N((a,b))))^{K/T} = S(\mft_N^*)^{W(G_{\gamma_N})}
\]
because the orbit space of the $T/T_N$-action on $T \gamma_N((a,b))$ is just $(a,b)$, which is contractible.
\end{proof}

\begin{prop}\label{lem:a} 
Let $N$ be a component of the non-abelian 1-skeleton.
\begin{enumerate}
\item[(a)] If $N$ is of type (I), 
 then the image of the restriction map 
 \[H^*_G(N)\to
H^*_G(N\cap \Mmax) =S(\mfg_{\gamma_N(0)}^*)^{G_{\gamma_N(0)}}\oplus S(\mfg_{\gamma_N(1)}^*)^{G_{\gamma_N(1)}}\]
consists of all pairs $(f,g)$
such that $f|_{\mfg_{\gamma_N}}=g|_{\mfg_{\gamma_N}}$.
\item[(b)] If $N$ is of type (II), then the image of the restriction map
\[H^*_G(N)\to
H^*_G(N\cap \Mmax) = S(\mfg_{\gamma_N(0)}^*)^{G_{\gamma_N(0)}}\]
consists of all $f$ for which there exists
$g\in S(\mfg_{\gamma_N(1)}^*)^{G_{\gamma_N(1)}}$ such that $f|_{\mfg_{\gamma_N}}=
g|_{\mfg_{\gamma_N}}$. 
\item[(c)] If $N$ is of type (III), then the image of the restriction map \[
H^*_G(N)\to 
H^*_G(N \cap \Mmax) = S(\mfg_{\gamma_N(0)}^*)^{G_{\gamma_N(0)}}
\]
 consists of all $f$ such that 
$f|_{\mfg_{\gamma_N}} = (\Ad_h^*f)|_{\mfg_{\gamma_N}}$, where $h\in G$ is such that $h\gamma_N(0)=\gamma_N(1)$.
\end{enumerate}
 \end{prop}
 
 \begin{proof} 
We prove (a) and (b) at the same time. Consider the covering 
\begin{equation}\label{eq:mvs0} 
N=(N\setminus G \gamma_N(1))\cup (N\setminus G \gamma_N(0)).
\end{equation} We claim first that $H^*_G(N\setminus G \gamma_N(1))\cong H^*_G(G \gamma_N(0))=S(\mfg_{\gamma_N(0)}^*)^{G_{\gamma_N(0)}}$. For that, let $\varepsilon>0$ be so small that for all $0<t<\varepsilon$, the isotropy group $G_{\gamma_N(t)}$ is equal to $G_{\gamma_N}$. We consider the covering of $N\setminus G \gamma_N(1)$ given by 
\begin{equation}\label{eq:mvs1}
N\setminus G \gamma_N(1) = U\cup V,
\end{equation}
where $U=G \gamma_N([0,\varepsilon))$ and $V=G \gamma_N((\varepsilon/2,1))$.  Then, $G \gamma_N(0)$ is a deformation retract of $U $ via the $G$-equivariant retraction $r_s(g\gamma_N(t))=g\gamma_N(st)$, $s\in [0,1]$. By Lemma \ref{lem:eqcohomgeodpiece}, both $H^*_G(V)$ and $H^*_G(U\cap V)$ are isomorphic to $S(\mft_N^*)^{W(G_{\gamma_N})}$, and hence isomorphic to each other. Thus, the equivariant Mayer-Vietoris sequence of the covering \eqref{eq:mvs1} becomes short exact
\[
0\longrightarrow H^*_G(N\setminus G \gamma_N(1))\longrightarrow H^*_G(U)\oplus H^*_G(V)\longrightarrow H^*_G(U\cap V)\longrightarrow  0
\]
and shows that the restriction map $H^*_G(N\setminus G \gamma_N(1))\to H^*_G(U)\cong S(\mfg_{\gamma_N(0)}^*)^{G_{\gamma_N(0)}}$ is an isomorphism. Using this, together with the analogous statement for $N\setminus G \gamma_N(0)$ and another application of Lemma \ref{lem:eqcohomgeodpiece}, this time to $G \gamma_N((0,1))$, we see that the equivariant Mayer-Vietoris sequence of the covering \eqref{eq:mvs0} takes the form
\[
\ldots \longrightarrow H^*_G(N) \longrightarrow S(\mfg_{\gamma_N(0)}^*)^{G_{\gamma_N(0)}}\oplus S(\mfg_{\gamma_N(1)}^*)^{G_{\gamma_N(1)}} \longrightarrow S(\mft_N^*)^{W(G_{\gamma_N})}\longrightarrow \ldots 
\]
This implies (a) and (b). 

For (c), we use the same method of proof, but this time with the covering 
\begin{equation}
N = G \gamma_N([0,\varepsilon)\cup (1-\varepsilon,1])\cup G \gamma_N((\varepsilon/2,1-\varepsilon/2)).
\end{equation}
By Lemma \ref{lem:eqcohomgeodpiece} and since $G\gamma_N(0)=G\gamma_N(1)$, the corresponding equivariant Mayer-Vietoris sequence takes the form
\begin{align}
\ldots \longrightarrow H^*_G(N) \longrightarrow
&S(\mfg^*_{\gamma_N(0)})^{G_{\gamma_N(0)}} \oplus S(\mft_N^*)^{W(G_{\gamma_N})}\nonumber\\
&\qquad\qquad\qquad \longrightarrow S(\mft_N^*)^{W(G_{\gamma_N})} \oplus S(\mft_N^*)^{W(G_{\gamma_N})} \longrightarrow \ldots \label{eq:longeqmvs}
\end{align}
Let $h\in G$ be such that $h\gamma_N(0)=\gamma_N(1)$. Then, $G_{\gamma_N(1)} = hG_{\gamma_N(0)}h^{-1}$ and hence, $\Ad_h^*:S(\mfg_{\gamma_N(0)})^{G_{\gamma_N(0)}}\to S(\mfg_{\gamma_N(1)})^{G_{\gamma_N(1)}}$ is an isomorphism.  The claim follows because an element $(f,g)$ contained in the middle term of \eqref{eq:longeqmvs} is mapped to $(f|_{\mft_N}-g,(\Ad_h^*f)|_{\mft_N}-g)$.
\end{proof}

Let us denote the components of $\Mmax$ by $O_A$, where $A$ is in a certain (finite) set  ${\mathcal A}$.

\begin{prop}\label{prop:b} The image of the canonical map 
$H^*_G(M_{r-1}) \to H^*_G(\Mmax) =\bigoplus_{A\in {\mathcal A}}H^*_G(O_A)$
consists of all $(f_A)_{A \in {\mathcal A}}$ with the property that:
\begin{itemize}
\item if there exist a component $N$   such that $N\cap \Mmax= O_A \cup O_B$,
for $A, B\in {\mathcal A}$, $A\neq B$, then
$(f_A, f_B)$ is in the image of the restriction map $H^*_G(N)\to H^*_G(O_A)\oplus H^*_G(O_B)$.
\item if there exist a component $N$ such that $N\cap \Mmax= O_A$,  
for $A\in {\mathcal A}$,
then
$f_A$ is in the image of the restriction map $H^*_G(N) \to H^*_G(O_A)$.
\end{itemize}
\end{prop}
 
\begin{proof}  From the long exact sequence of the pair $(M_{r-1}, \Mmax)$ we deduce that
the image of the map $H^*_G(M_{r-1})\to H^*_G(\Mmax)$ is equal to the kernel of
$H^*_G(\Mmax) \to H^{*+1}_G(M_{r-1},\Mmax)$. We set
  $$U:=\bigcup_{N} G\gamma_N([0, \varepsilon)) \cup G\gamma_N((1- \varepsilon, 1])$$
where $N$ runs through the components of the non-abelian $1$-skeleton of the $G$-action, and $\varepsilon>0$ is chosen so small that $\Mmax$ is a deformation retract of $U$ (see the proof of Lemma \ref{lem:a}). 
  Moreover, $U$ is open in $M_{r-1}$: the complement of $U$ in 
  $M_{r-1}=\bigcup_{N} G\gamma_N([0,1])$ is equal to
$\bigcup_{N}G\gamma_N([ \varepsilon, 1- \varepsilon])$, which is a compact space. 

By using the homotopy, respectively the excision property we obtain: 
$$H^{*+1}_G(M_{r-1}, \Mmax)= H^{*+1}_G(M_{r-1}, U) = H^{*+1}_G(M_{r-1}\setminus \Mmax, U\setminus \Mmax).$$
Since $M_{r-1}\setminus \Mmax$ is equal to the {\it disjoint} union of all $N \setminus \Mmax$, the last term in the
previous equation is equal to $\bigoplus_{N} H^{*+1}_G(N\setminus \Mmax, G\gamma_N((0, \varepsilon))\cup G\gamma_N((1-\varepsilon,1)))$.
By using again the excision, respectively homotopy property, the last sum is equal to 
$$\bigoplus_{N} H^{*+1}_G(N, N\cap \Mmax) .$$The result stated by the proposition follows now
from the long exact cohomology
sequence of the pairs $(N, N\cap \Mmax)$. 
\end{proof}

\section{Non-abelian GKM graphs}\label{sec:non-abelianGKMgraphs}
\subsection{Abstract non-abelian GKM graphs}
 Recall that Propositions \ref{lem:a} and \ref{prop:b}  give a complete description of the 
$G$-equivariant cohomology of $M$ associated to the group action
$G\times M \to M$ which satisfies the non-abelian GKM conditions.
The main observation of this section is that the output, i.e., 
$H^*_G(M)$ as an $S(\mfg^*)^G$-algebra, can be encoded into 
 the so-called non-abelian GKM graph
associated to the action. This is a special kind of graph, which can be abstractly defined as follows:
\begin{dfn} A {\it non-abelian GKM graph}  (associated to $G$) is a diagram which consists of:
\begin{itemize}
\item finitely many circles (sometimes just simple closed contours), indexed 
by  $A \in {\mathcal A}$, whose interiors are pairwise disjoint 
\item finitely many dots, each dot $a$ lying inside a circle and  labeled
with a closed subgroup  $G_a \subset G$ 
\item finitely many stars, lying outside the circles and which are labeled with closed subgroups of $G$ as well
\item edges that join pairs of dots or a dot with a star; each
edge is labeled with a group  which is contained in the groups corresponding
to the two end-points
\item arrows that join two dots; an arrow going from $a$ to $b$, where
$a$ and  $b$ are dots, is labeled with an isomorphism $\mfg_{a} \to \mfg_{b}$.
\end{itemize} 
 Each circle contains a distinguished dot, the {\it representative}, which is
joined with any other dot inside that circle by exactly one arrow emerging at the representative. There are no other
arrows.
\end{dfn}
Here comes the precise construction.

\subsection{Construction of the non-abelian GKM graph}

The input consists of the components $N$ of the non-abelian 1-skeleton $M_{r-1}$;
each of them comes equipped with
 an adapted geodesic segment $\gamma_N$.

\noindent {\it Step 1.} We draw one big circle for each connected component $O_A$ of $\Mmax$. Inside the circle we draw one dot, which corresponds to  an endpoint 
contained in $O_A$ of one of the chosen adapted geodesic segments.
We call this point a 
\emph{representative}  of  $O_A$. We attach to it the label consisting of the $G$-isotropy group of that point. In the following steps there will potentially be drawn more dots inside the circle, corresponding to other endpoints contained in $O_A$
of adapted geodesic segments.

\noindent {\it Step 2.} For each component $N$ of type (I) we proceed in the following way: 
Say that the two orbits in $\Mmax$ contained in $N$ are $O_A$ and $O_B$.
At this step, we attach to $N$ a line segment and possibly one or two arrows as follows. If one of the two endpoints of $\gamma_N$   
 does not correspond yet to a dot inside the corresponding circle, then we add an additional dot inside this circle, label it with the  isotropy group of that endpoint, and draw an arrow with tail at the representative of $O_A$ and tip at the new dot.
 Say that the representative and the new dot correspond to the points 
 $p_a$ and $hp_a$, where $h\in G$, and thus carry the labels $G_{p_a}$, respectively $G_{hp_a}=
hG_{p_a}h^{-1}$: then
the arrow will carry the label $\Ad_{h}$. (More precisely, the conjugation map $c_{h}: g\mapsto hgh^{-1}$ is an isomorphism
between $G_{p_a}$ and $G_{hp_a}$, hence $\Ad_h$, i.e., its differential at $e$, is a linear isomorphism  
between the Lie algebras $\mfg_{p_a}$ and $\mfg_{hp_a}$.)
The same procedure is used for $O_B$, i.e., we may need to draw an arrow
inside the corresponding circle as well.
We next join the two endpoints of $\gamma_N$  by a line
segment  and label it with  $G_{\gamma_N}$. Note that, by Lemma \ref{lem:adapted}, we may have translated
$\gamma_N$ by an element of $G$ in such a way that 
 at least one of its two endpoints 
  is the representative of its orbit in $\Mmax$ and avoid in this way 
  drawing the corresponding arrow.

\noindent {\it Step 3.} Here, we consider the  components of type (III).
For any such component, the same procedure as in Step 2 applies, the only difference being that the two dots that are connected by the line segment are contained in the same circle, because the endpoints of the corresponding adapted geodesic segment are contained in the same orbit in $\Mmax$. As before, we may have translated
this geodesic segment to make one of its endpoints coincide with  the representative 
of its orbit.

\noindent {\it Step 4.} Finally, we treat the components of type (II). For any such component $N$, we first draw a star outside all circles and label it with the isotropy group 
 $G_{\gamma_N(1)}$. We then connect it by a line segment with the dot corresponding to $\gamma_N(0)$ (if this dot does not exist yet, then introduce it, together with an arrow, as in Step 2), and label that line segment with  $G_{\gamma_N}$. 
Again, we may first replace $\gamma_N$ by a $G$-translate of it in such a way that
$\gamma_N(0)$ becomes the representative of its orbit.

\subsection{The non-abelian GKM theorem}
To any non-abelian GKM graph  we associate an $S({\mfg }^*)^G$-algebra as follows:
First for each circle $A$ we set $G_A:=G_a$, where $a$ is the representative of $A$.
Then we consider the subalgebra of $\bigoplus_{A} S(\mfg_{A}^*)^{G_{A}}$ (where the $S(\mfg^*)^G$-algebra structure is given by the natural restriction maps) consisting of all tuples 
$(f_A)$ which satisfy the following two types of relations: 
\begin{itemize}
\item  Consider an edge whose endpoints are  the dots $a$ and $b$.
We take the (possibly equal)  circles $A$  and $B$ 
in which  $a$, respectively  $b$ are contained, as well as the  arrows
from their representatives to $a$, respectively $b$,  labeled by the isomorphisms 
$h_{Aa}: \mfg_{A} \to \mfg_{a}$,
respectively $h_{Bb}:\mfg_{B}\to \mfg_{b}$. (If $a$ or $b$ is the representative of its circle, then we set $h_{Aa}$ or $h_{Bb}$ to be the identity.)
If $H$ is the label of the edge and $\mfh$ its Lie algebra, then we must have
$$ f_A\circ h_{Aa}^{-1}|_{\mfh}  =  f_B\circ h_{Bb}^{-1}|_{\mfh}.$$ 
\item  Consider an edge, whose endpoints are  the dot $a$ and a certain star.
We take the circle $A$  
in which  $a$ is contained and also the  arrow
from its representative to $a$,  labeled by the isomorphism 
$h_{Aa}: \mfg_{A}\to \mfg_{a}$. (If $a$ is the representative, then we set $h_{Aa}$ to be the identity.)
If $K$ is the label of the star, $H$  the label of the edge, and $\mfk$,
respectively
$\mfh$ their Lie algebras, then we must have
$$ f_A\circ h_{Aa}^{-1}|_{\mfh}  =  g|_{\mfh},$$
for some $g\in S(\mfk^*)^K$. 
\end{itemize}

In the previous subsection we have attached  a non-abelian GKM graph to a group
action of $G$ on $M$ which satisfies the GKM conditions.
From Propositions \ref{lem:a} and \ref{prop:b}  we immediately deduce:

\begin{thm}\label{via} If the group action of $G$ on $M$ satisfies the GKM conditions, then
  $H^*_G(M)$ is isomorphic to the $S(\mfg^*)^G$-algebra induced by
 the non-abelian GKM graph attached  to the action.
 \end{thm}

In what follows we will give some applications of this theorem to certain concrete
situations.

\begin{ex}\label{sp22} 
Let $G:=\Sp(1) \times \Sp(1)$ be the subgroup of all diagonal matrices in $\Sp(2)$, and $T:=S^1\times S^1$, where $S^1=\{ a+bi \in \HH \mid a, b \in \RR, \ a^2+b^2=1\}$. 
The torus $T$ is clearly maximal in both $\Sp(2)$ and $\Sp(1)\times \Sp(1)$. We are interested
in the action of $G:=\Sp(1)\times \Sp(1)$ on the quotient $M:=\Sp(2)/T$. The latter space is a principal
orbit of the adjoint representation of $\Sp(2)$. It is the orbit of a regular
element in $\mft$, the Lie algebra of $T$. The fixed point set $M^T$ is the same as the
intersection of the orbit with $\mft$, and actually equal
to a Weyl group orbit.  The situation is described in Figure \ref{fig4}: the Weyl group
$W(\Sp(2))$ is generated by the reflections through the  lines $l_1$, $l_2$, $l_3$, and $l_4$, whereas
$W(\Sp(1)\times \Sp(1))\simeq \ZZ_2 \times \ZZ_2$ 
is its subgroup generated by the reflections through $l_1$ and $l_3$.
 \begin{figure}[htb]
 \includegraphics[width=7cm]{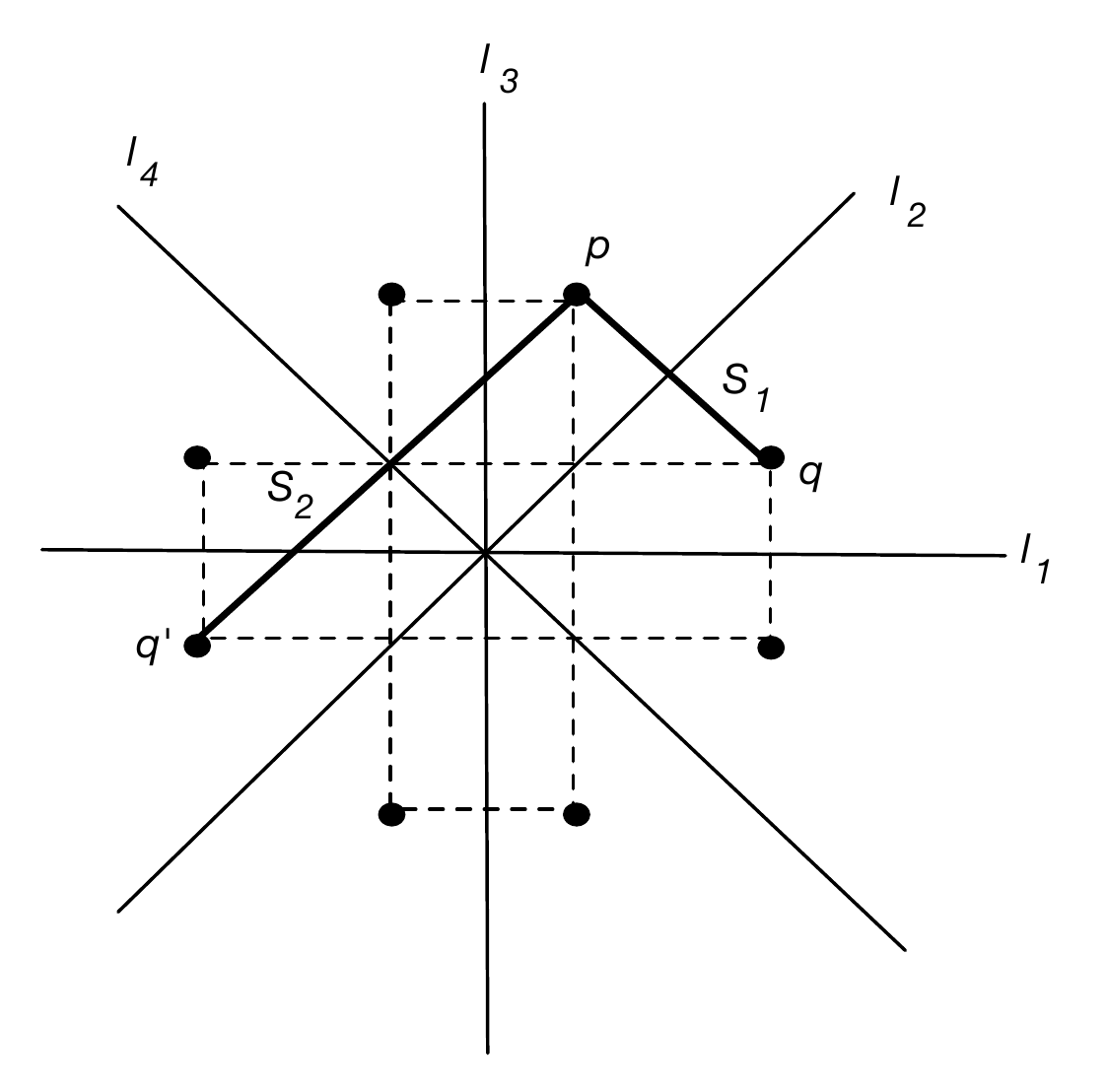}
 \caption{}
 \label{fig4}
\end{figure}

The corners of the two rectangles in the diagram lie respectively in the same $\Sp(1)\times \Sp(1)$-orbit.
That is, the set $\{Gx \mid  x\in M^T\}$ has two elements, which are $Gp$ and $Gq$.
There are sixteen $T$-invariant spheres $S_i$ in $\Sp(2)/T$, which join pairs of the type $x, sx$, where
$x$ is one the eight $T$-fixed points and $s$ one of the reflections in $W(\Sp(2))$.
 However, the set 
$\{G S_i \mid 1\le i \le 16\}$ has only two elements, which are induced by the
spheres $S_1$ and $S_2$ represented in the diagram.  The intersection of $GS_1$ with $GS_2$
is the union $Gp\cup Gq$.  

The GKM graph therefore has two line segments, induced by $GS_1$ and $GS_2$. 
We have $S_1\cap \Mmax=\{p, q\}$ and $S_2\cap \Mmax =\{p, q'\}$. 
As representatives of the two orbits $Gp$ and $Gq$ we choose $p$ and $q$.
The $G$-isotropy groups of $p, q$, and $q'$ are $T$. Recall that $T$ is
by definition a direct product $S^1 \times S^1$. One can see that the principal $T$-isotropy groups along
the spheres $S_1$ and $S_2$ are $\Delta:=\{ (z, z) \mid z\in S^1\}$, respectively
$\Delta':=\{(z, z^{-1}) \mid z\in S^1\}$. The Lie algebra
$\mft$ is the direct sum of  two copies of $\RR$, which is the Lie algebra of $S^1$.
In this way we obtain a system of coordinates $(t_1, t_2)$ on $\mft$, such that
the lines $l_1, l_2, l_3,$ and $l_4$ are described by the equations
$t_2=0$, $t_1-t_2=0$, $t_1=0$, respectively $t_1+t_2=0$.  
We have $Gq'=Gq$, so the only thing we still need in order to complete the graph is
to find $g_0\in G$ such that $q'=g_0q$ and then determine the differential of the map
$c_{g_0}: T\to T$. In fact, $g_0$ is in the $G$-normalizer of $T$ and the differential
of $c_{g_0}$ at $e$ is just the composition $s_1s_3$ (where $s_i$ is the reflection through $l_i$), that is, the
map $(t_1, t_2)\mapsto (-t_1, -t_2)$. The GKM graph is given in Figure \ref{fig5}. 

\begin{figure}[htb]
 \includegraphics[width=5cm]{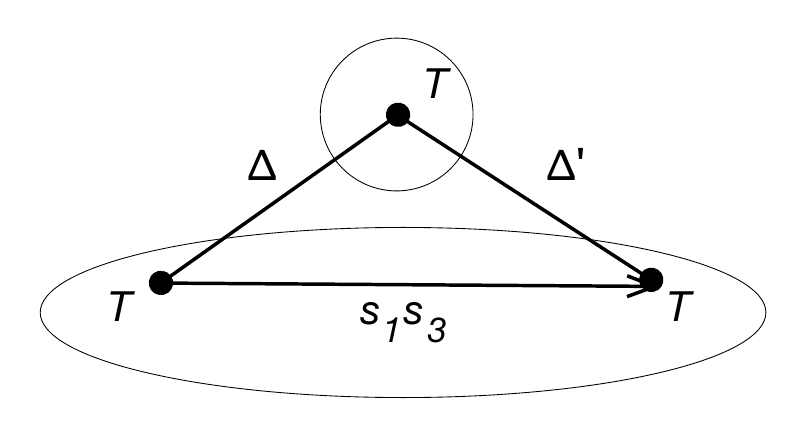}
 \caption{}
 \label{fig5}
\end{figure}

Thus  $H^*_{\Sp(1)\times \Sp(1)}(\Sp(2)/T)$ is the subalgebra of $\RR[t_1, t_2]\oplus \RR[t_1, t_2]$
consisting of all pairs $(f_1, f_2)$ with the following two properties:
\begin{itemize}
\item $f_1(t_1, t_2) -f_2(t_1, t_2)$ is divisible by $t_1-t_2$
\item $f_1(t_1, t_2)-f_2(-t_1,-t_2)$ is divisible by $t_1+t_2$
\end{itemize}
\end{ex}

\begin{ex}\label{u2} Let us consider again the $\U(2)$-action on $\HH P^1$ which has been defined in Example \ref{ex:hp1}. 
One can distinguish two $T$-invariant spheres in $\HH P^1$ between the two $T$-fixed points 
$[1:0]$ and $[0:1]$. To describe them, we consider the splitting $\HH = \CC + \CC j$.
The two spheres  are 
\begin{align*}{}&S_1:=\{[z_1:z_2] \mid z_1, z_2 \ {\rm are \ in \ } \CC, {\rm \ not \ simultaneously \ zero}\},\\
{}&S_2:= \{[z_1: z_2j] \mid z_1, z_2 \ {\rm are \ in \ } \CC, {\rm \ not \ simultaneously \ zero}\}.\end{align*}
Note that $S_1$ is in fact  the $\U(2)$-orbit of $[1:0]$.
The tangent spaces to $S_1$ and $S_2$ at $[1:0]$ are the direct summands of  the splitting of $T_{[1:0]}\HH P^1$
into irreducible $T$-spaces;
 the corresponding  representations of $T$ are given by
 \begin{align*}
{}& (z_1, z_2).v = z_1z_2v, \quad z_1, z_2\in S^1, v\in \CC\simeq T_{[1;0]}S_1\\
{}& (z_1, z_2).v= z_1z_2^{-1}v,  \quad z_1, z_2\in S^1, v\in \CC\simeq T_{[1;0]}S_2.
\end{align*}
We consider the geodesic $\gamma(t):=[\cos t; (\sin t)j]$, $0\le t \le \pi/2$, which joins
$[1:0]$ with $[0:1]$, is $\U(2)$-transversal, and is contained in $S_2$.
For small $t>0$, the $\U(2)$-isotropy group at $\gamma(t)$ is the subgroup $\Delta'$ of $T$ consisting 
of all $(z, z^{-1})$, with $z\in \CC$, $|z|=1$. 
To construct the GKM graph, we only need to observe that $\U(2)S_1=\U(2).[1:0]$,
which  is just an orbit; hence the graph has just one line segment, which corresponds to
$\U(2)S_2$.
The $\U(2)$-isotropy group at the midpoint $\gamma(\pi/4)$
is equal to $\SU(2)$, canonically embedded in $\U(2)$. 
Hence $S_2$ fits into the situation described by Lemma \ref{gsi} (b),
which means that $\U(2)S_2$
is a  component of type (II) of the non-abelian 1-skeleton. (In fact, we have $\U(2)S_2=\HH P^1$, so it is the only component of the non-abelian 1-skeleton.)
The non-abelian GKM graph is presented in  Figure \ref{fig6}.
 
 \begin{figure}[htb]
 \includegraphics[width=6cm]{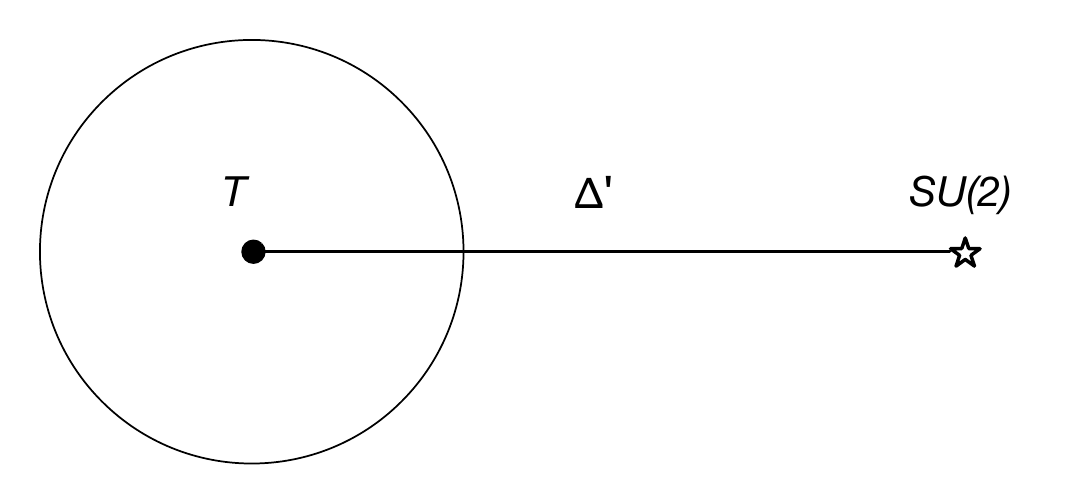}
 \caption{}
 \label{fig6}
\end{figure}

A simpler way to deduce this is by simply observing  that the action we investigate here 
is the cohomogeneity-one action corresponding to the group diagram
$\U(2)\supset T, \SU(2) \supset \Delta'$, see Remark  \ref{coh1} below, especially
equation (\ref{cohco}).

Let us now calculate explicitly $H^*_{\U(2)}(\HH P^1)$. To this end, we choose the coordinates
$(t_1, t_2)$ on the Lie algebra $\mft$ of $T$ which arise naturally from writing $T=S^1\times S^1$.
The Lie algebra of $\Delta'$ is isomorphic to $\RR$, which is embedded in
$\mft$ via $x\mapsto (x, -x)$. The same $\RR$ is embedded into the Lie algebra $\mfs\mfu(2)$ of
$\SU(2)$ via $x\mapsto {\rm Diag}(x, -x)$. Its image in $\mfs\mfu(2)$ is actually the Lie algebra of the standard
maximal torus in $\SU(2)$. Thus $H^*(B\SU(2))=S(\mfs \mfu(2)^*)^{\SU(2)}$ consists of all
$f\in \RR[t]$ which are fixed by the Weyl group of $\SU(2)$. But this group is isomorphic to
$\ZZ_2$, which acts on $\RR$ via $x\mapsto -x$ and consequently on $\RR[t]$ via $f(t)\mapsto f(-t)$.
In conclusion, the abovementioned cohomology algebra consists of all $f(t_1, t_2) \in \RR[t_1, t_2]$ with the property that
 $f(t, -t)=f(-t, t)$; equivalently, 
 $$H^*_{\U(2)}(\HH P^1)\simeq \{f\in \RR[t_1, t_2] \mid f(t_1, t_2)-f(t_2, t_1) {\rm \   is \  divisible \  by} \ t_1+t_2\}.$$ 
 \end{ex}

\begin{rem}\label{coh1} The group action in the previous example is of cohomogeneity-one. 
Consider the general situation when a compact connected Lie group $G$ acts  on a compact connected manifold $M$ with cohomogeneity equal to one, such that $M/G$ is homeomorphic to a closed interval. 
Let $G\supset K_+,K_-\supset H$ be the associated group diagram, see,
e.g., \cite[Section1]{Z}. Assume that the group action satisfies the non-abelian GKM conditions. The condition of being equivariantly formal is equivalent to at least one of the groups $K_+,K_-$ having rank equal to the rank of $G$, see \cite[Corollary 1.3]{GM}. 
As the 1-skeleton coincides with the whole $M$, the GKM graph consists of only one line segment, which either connects two dots (if both $K_+$ and $K_-$ are of maximal rank) or a dot and a star (if only one of $K_+$ and $K_-$ is of maximal rank). In the first case, we have 
\begin{equation}\label{cohcoh}
H^*_G(M) = \{(f,g)\in S(\mfk_+^*)^{K_+} \oplus S(\mfk_-^*)^{K_-}\mid f|_{\mfh} = g|_{\mfh}\},
\end{equation}
and in the second (assume that $K_+$ is of maximal rank) we have
\begin{equation}\label{cohco}
H^*_G(M) = \{f\in S(\mfk_+^*)^{K_+}\mid f|_{\mfh} = g|_{\mfh} \text{ for some } g\in S(\mfk_-^*)^{K_-}\}.
\end{equation}
Note that equation \eqref{cohcoh} holds true for any cohomogeneity-one action which satisfies $\rk G = \rk K_+ = \rk K_- = \rk H + 1$, even if  the
GKM conditions are not satisfied, see \cite[Corollary 4.2]{GM}. For example,
consider the $\SO(3)$-action on the sphere $S^4$ which is induced by the group diagram 
$\SO(3) \supset {\rm S(O}(2)\times {\rm O}(1)), {\rm S(O}(1) \times {\rm O}(2)) \supset \ZZ_2\times \ZZ_2$, 
see, e.g., \cite[Section 2]{Z}; 
the third GKM condition is obviously not satisfied by this action. Similarly,  \eqref{cohco} holds true whenever $\rk G = \rk K_+=\rk K_-+1 = \rk H + 1$,  (to perceive this, combine \cite[Corollary 4.2]{GM} with the injectivity of $H^*_G(M)\to H^*_G(G/K^+)$). 
\end{rem}

\begin{ex}\label{typecc} We now consider again the action of $K$ on $M$ which is described in Example \ref{typec}. First, all $T$-fixed points are on the same $K$-orbit,
which means that $\Mmax$ consists of exactly one component.  
The locus $M_1$ of all points in $M$ of isotropy corank at least 1 consists
only of $KS$. Indeed, the 1-skeleton of the $T$-action on $M$ is the union of
fifteen 2-spheres among which only six have the $K$-orbit different from the
$K$-orbit of a point. Moreover, the $K$-orbit of any of those six spheres is equal to $KS$:
the reason is that the two $T$-fixed points of such a sphere can be mapped to $p$, respectively
$q$ by an element of $W(K)$; we then use \cite{GHZ}, namely the last paragraph in Section 2.2.5 
and Theorem 2.5. 
To construct the graph, we only need to perform Step 3 once.
The $K$-stabilizers of $p$ and $q$ are both equal to $T$. 
Also note that  $S$ is a component of the fixed point set of the subtorus
$T_\alpha$ of $T$ whose Lie algebra is $\alpha^\perp$. Finally, observe that
$q=s_{\alpha_2} \circ s_{\alpha_1}p$.  The GKM graph is presented in the figure below.
\begin{figure}[htb]
 \includegraphics[width=4.5cm]{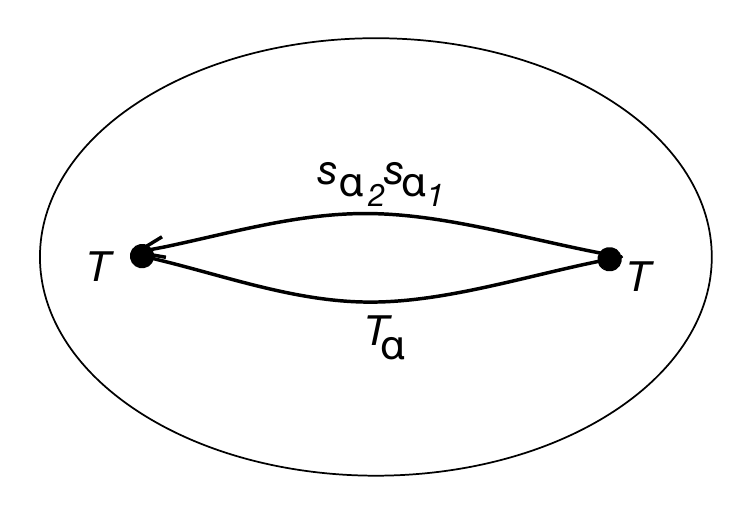}
\caption{}
 \label{fig7}
\end{figure}

We obtain the following presentation:
$$H^*_K(M)\simeq \{f\in S(\mft^*) \mid 
f-f\circ s_{\alpha_1}\circ s_{\alpha_2}{\rm \ 
is \  divisible \  by  \ } \alpha\}.$$
\end{ex}

\begin{rem}\label{simpler} 
In the general context of a non-abelian GKM action of $G$ on $M$, 
one can also compute
the 
corresponding $G$-equivariant cohomology algebra by taking a maximal
torus $T\subset G$ along with the corresponding Weyl group $W(G)$
and using the formula $H^*_G(M)=H^*_T(M)^{W(G)}$.
The result stated in Lemma \ref{ift} is crucial here: 
by the classical GKM theorem, see, e.g., Theorem \ref{thm:gkmclassical},  the algebra $H^*_T(M)$
can be expressed as a subalgebra of $\oplus_{p\in M^T}S(\mft^*)$, which we 
refer to as the GKM algebra. 
The Weyl group $W(G)$ acts canonically on $M^T$ and the
action of $W(G)$ on the GKM algebra is given by
$$w.(f_p)_{p\in M^T}=(g_p)_{p\in M^T}, \ {\rm where} \ 
g_p=f_{w^{-1}p}\circ w^{-1} \ {\rm for \ all \ } p\in M^T.$$
 We conclude that $H^*_G(M)$ consists of all $(f_p)_{p\in M^T}$
 in the GKM algebra with the property that 
 \begin{equation}\label{fwp} f_{wp}=f_p\circ w^{-1},\end{equation} for all
 $p\in M^T$ and all $w\in W(G)$. It is interesting to compare this presentation
 with the one obtained by using our method. For instance, the equivalence is easy to establish in Examples \ref{sp22} and \ref{u2} and we leave it as an exercise for the reader.  Things seem to be different in Example \ref{typecc}: it is not immediately
 visible that fifteen
 divisibility relations involving six polynomials $f_p$, combined with
   the $W(K)$-invariance conditions (\ref{fwp})  lead to just one divisibility relation.
      \end{rem}

\end{document}